\documentclass[10pt]{amsart}
\usepackage{latexsym,amsmath,amssymb}
\usepackage{mathrsfs}
\usepackage{mathabx}
\usepackage{color}
\usepackage{indentfirst}
\renewcommand{\epsilon}{\varepsilon}

\newtheorem{theorem}{Theorem}[section]
\newtheorem{Lemma}{Lemma}[section]
\newtheorem{proposition}{Proposition}[section]

\numberwithin{equation}{section}

\newcommand{\bth}{\begin{theorem}}
	\newcommand{\ble}{\begin{lemma}}
		\newcommand{\bcor}{\begin{corr}}
			\newcommand{\bdeff}{\begin{deff}}
				\newcommand{\bprop}{\begin{proposition}}
					\def\be{\begin{equation}}
						\def\ee{\end{equation}}
					\def\bt{\begin{theorem}}
						\def\et{\end{theorem}}
					\def\ba{\begin{array}}
						\def\ea{\end{array}}
					\def\bl{\begin{lemma}}
						\def\el{\end{lemma}}
					
					\newcommand{\ele}{\end{lemma}}
				\newcommand{\ecor}{\end{corr}}
			\newcommand{\edeff}{\end{deff}}
		
		\newcommand{\eprop}{\end{proposition}}

	\renewcommand{\Pi}{\varPi}

	\renewcommand{\epsilon}{\varepsilon}
	\newcommand{\sgn}{{\text {sgn}}}

\title[Global well-posedness of fractional Schr\"odinger equations] {Global well-posedness of cubic fractional Schr\"odinger equations in one dimension}

\date{\today}

\begin{document}
	\maketitle

	\centerline{
		\author{
			Huali Zhang
			\footnote{
				Department of Mathematical Sciences, Changsha University of Science and Technology, Changsha 410114, China.
				{\it Email: huali.zhang@aei.mpg.de}}
		}
		\and
		  Shiliang Zhao
		  \footnote{ School of Mathematical Sciences, Sichuan University, Chengdu 610064, Peoples Republic of China.
		 {\it Email: zhaoshiliang@scu.edu.cn}
		 }
	}
		
\begin{abstract}
In this paper, we consider the Cauchy's problem of global existence and scattering behavior of small, smooth, and localized solutions of cubic fractional Schr\"odinger equations in one dimension,
\begin{equation*}
		\mathrm{i} \partial_t u- (-\Delta)^{\frac{\alpha}{2}} u=c_*|u|^2u,
\end{equation*}
where $\alpha \in (\frac{1}{3},1), c_* \in \mathbb{R}$. Our work is a generalization of the result due to Ionescu and Pusateri \cite{IP}, where the case $\alpha=\frac{1}{2}$ was considered. The highlight in this paper is to give a modified dispersive estimate in weighted Sobolev spaces for cubic fractional Schr\"odinger equations, which could be used for $ \alpha \in (\frac{1}{3},1)$. Based on this modified dispersive estimate, we prove the global existence and modified scattering behavior of solutions combining space-time resonance and bootstrap arguments. 
\end{abstract}
		


	
\section{Introduction and Main results}
In this paper we study the Cauchy's problem of
nonlinear fractional Schr\"odinger equations
\begin{eqnarray}\label{NFSE}
	\left
	\{
	\begin{array}{l}
		\mathrm{i} \partial_t u- (-\Delta)^{\frac{\alpha}{2}} u=c_*|u|^2u, \ (t,x) \in \mathbb{R}^+ \times \mathbb{R},\\
		u|_{t=0}=u_0(x),\\
		
	\end{array}
	\right.
\end{eqnarray}
where $u: \mathbb{R}^+ \times \mathbb{R}\rightarrow \mathbb{C}$, \ $u_0(x): \mathbb{R}\rightarrow \mathbb{C}$
 and $ c_* \in \mathbb{R}$.

The model \eqref{NFSE} was shown by Laskin \cite{Laskin} in quantum mechanics. Namely, the quantum mechanics path integral over Brownian trajectories leads to 
classical Schr\"odinger equations ($\alpha=2$), and the path integral over
L$\mathrm{\acute{e}}$vy trajectories leads to  fractional Schr\"odinger equations ($1<\alpha<2$).

The Cauchy's problem for nonlinear classical Schr\"odinger equations has been studied extensively. For instances, to prove the global existence of Cauchy's problems for a class of classical Schr\"odinger equations
\begin{equation}\label{1111}
	\mathrm{i} \partial_t u- \Delta u=|u|^{p-1}u, \ (p>1),
\end{equation}
there are basically two approaches, namely, Strichartz estimates and vector-fields methods. In the aspect of Strichartz estimates,  the dacay estimate of linear part possesses fundamental importance. For $p$ is large enough, 
one can use Strichartz estimates to prove the global existence of small solutions; for small $p$, the structure of nonlinear terms starts to play a crucial role. More precisely, two values of $p$ are particularly important--the Strauss exponent $\frac{\sqrt{n^2+12n+4}+n+2}{2n}$ \cite{Strauss} and the short range exponent $1+\frac{2}{n}$. For $p$ is larger than the Strauss exponent, it has been proved by Strauss \cite{Strauss} that global solutions exist of  \eqref{1111} for small initial data, and furthermore, the solution scatters in a certain sense. For $p>1+\frac{2}{n}$, the solution becomes asymptotically free, and wave operators can be
constructed in general for small data in \cite{GOV,Nakanishi}. 
But the situation is quite different for $1<p<1+\frac{2}{n}$, and we refer the readers to see the interesting work \cite{Ozawa, Haya1} of Ogawa and Hayashi-Naumkin respectively. Meanwhile, McKean and Shatah in \cite{MS} proved the global existence of solutions to \eqref{1111} by using vector-fields methods. The key in their paper is to prove global Sobolev inequalities, which is similar to that for the wave equation obtained in \cite{K} by Klainerman. 
For $p=1+\frac{2}{n}$, Germain etc. in \cite{GMS2} proved the global existence and scattering of small solutions for 2D quadratic Schr\"odinger equations using the idea of space-time resonance, where the Fourier analysis they developed is essentially a new point. For other nonlinearities relying on $u,\bar{u},\nabla u, \nabla \bar{u}$ to \eqref{1111}, there are some known results addressing global existence and asymptotic behavior of solutions. The readers could see \cite{CW,GMS,HN,Kawahara,C,D,HMP,W,Z} for referring.

However, concerning the Cauchy's problem of linear fractional Schr\"odinger equations
\begin{equation}\label{lfse}
	\mathrm{i} \partial_t u- (-\Delta)^{\frac{\alpha}{2}} u=F(t,x), \quad (t,x) \in \mathbb{R}^+ \times \mathbb{R}^n,
\end{equation}
the vector-fields method for classical Schr\"odinger equations may not be used directly and Stricharz estimates are also not apparent. Despite that, there are some key developments for nonlinear fractional Schr\"odinger equations. For instances, Guo and Huo in \cite{GH} showed that global solutions to \eqref{NFSE} exist in one dimension for $1<\alpha <2$, where tri-linear estimates in Bourgain space played a important role. Later, Guo etc. in \cite{GSWZ} proved the global
existence of radial solutions  in critical energy spaces to \eqref{lfse} ($\alpha \in (\frac{n}{n-1},1), n\geqslant2$) with power-type nonlinearity. Recently, Cho etc. in \cite{Cho}  proved inhomogeneous Strichartz estimates for \eqref{lfse} ($n \geqslant 2$) in the radial case, and they applied this to establish the global well-posedness theory for a class of potential nonlinear fractional Schr\"odinger equations. 
There are also some other interesting results for nonlinear fractional Schr\"odinger equations, see references 
\cite{GqZ,BHL,GHX} 
for details.

For Cauchy's problem \eqref{NFSE}, we are asking for which $\alpha$ we can obtain a global solution in some Sobolev spaces, given that initial data satisfy smallness conditions.
Since the pointwise decay is just $t^{-\frac{1}{2}}$, it seems that one could only prove almost global existence of small solutions. Ionescu and Pusateri presented a pioneer work for us in \cite{IP}. Establishing dispersive linear estimates and identifying a suitable nonlinear logarithmic correction, they proved the global existence and scattering of small, smooth solutions of \eqref{NFSE} for $\alpha=\frac{1}{2}$.
However, the long time behavior of solutions to \eqref{NFSE} is not very clear for $\alpha \in (0,\frac{1}{2})\cup (\frac{1}{2},1)$. In case $\alpha \in (0,\frac{1}{2})\cup (\frac{1}{2},1)$, we could not get the conclusion of global small solutions if we use directly dispersive estimates devised in \cite{IP}. Those push us to establish appropriate dispersive linear estimates to equations \eqref{NFSE} for general $\alpha$. Considering different value of $\alpha$, the stationary point of phase is different. At the same time, we also need enough time decay rate in dispersive estimates. Therefore, how to set the time decay rate for general $\alpha$ plays a crucial role. Exploring stationary phase methods deeply, we give a modified dispersive estimate (see Lemma 3.1), which is suitable for all $\alpha\in (\frac{1}{3},1)$. We then use it to prove the global existence of small solutions for \eqref{NFSE} by bootstrap arguments and energy estimates. In this process, the $Z$ norm(the definition of $Z$ see \eqref{Z}) of middle frequency of solution is the most difficult part (see \eqref{41} and Lemma \ref{reduction}). Considering nonlinear structures of \eqref{NFSE}, we could go through these difficulties using space-time resonance analysis. Here, the concrete Fourier analysis techniques is a developing idea of Germain-Mausmoudi-Shatah and Ionescu-Pusateri. 

The goal of this paper is to prove the global existence and scattering behavior of small solutions for cubic fractional Schr\"odinger equations in one dimension for $\alpha \in (\frac{1}{3},1)$. 
Our result is as follows.
\begin{theorem}\label{GFS}
	Let $\alpha \in (\frac{1}{3},1)$ and $N_0=100$. For any fixed $\alpha$, we choose $p_0 \in (0, \frac{1}{1000}]$ such that $\alpha \leq \frac{1}{1+2p_0}$.

Assume $0<\varepsilon_0<1$ and $u_0 \in H^{N_0}(\mathbb{R})$ satisfying
	\begin{equation}
		||u_0||_{H^{N_0}(\mathbb{R})}+||x \cdot \partial u_0||_{L^2(\mathbb{R})}+||u_0||_{Z} \leq \varepsilon_0,
	\end{equation}
where
	\begin{equation}\label{Z}
	Z:=\big\{g\in {H^{4}(\mathbb{R})}:||g||_Z=||(1+|\xi|)^{10} \hat{g}(\xi)||_{L_{\xi}^{\infty}} <\infty \big\}.
	\end{equation}
	Then there is a unique global solution $u \in C([0,+\infty), H^{N_0}(\mathbb{R}))$ of Cauchy's problem \eqref{NFSE}.

	In addition, let $f(t)=e^{\mathrm{i}t(-\Delta)^{\frac{\alpha}{2}}}u(t)$. Then there exists a universal positive constant $C$ such that the following bound
	\begin{equation}
		\sup_{t\geq 0}\big\{(1+t)^{-p_0}||f(t)||_{H^{N_0}(\mathbb{R})}+(1+t)^{-p_0}||x \cdot \partial f||_{L^2(\mathbb{R})}\big\}\leqslant C \varepsilon_0
	\end{equation}
hold. Moreover, the solution possesses the scattering behavior: there is $p_1>0$ and $w \in L^\infty$ with the property that 
	\begin{equation}
		\sup_{t \in [0,\infty)}\left(1+t\right)^{p_1}||(1+|\xi|)^{10}e^{iH(\xi,t)} \hat{f}(\xi,t)-w(\xi)||_{L_{\xi}^{\infty}} \leqslant C \varepsilon_0,
	\end{equation}
where
\begin{equation}\label{H}
H(\xi,t):=c_0c_*|\xi|^{2-\alpha}\int^t_0|\hat{f}(\xi,r)|^2\frac{dr}{r+1}, \ c_0=\frac{2\pi}{\alpha(1-\alpha)}.
\end{equation}
\end{theorem}

\textbf{Notations:}
	The notation $A \lesssim B$ means $A \leq CB$ for some universal
	constant $C$ independent with $A,B$.  We also denote $k^+:=\max\{0,k\}$ for integer $k$.

The plan of this paper is as follows: in Section 2, we prove Theorem 1.1 as a result of a bootstrap argument based on the local existence and dispersive estimate. In Section 3, we prove the dipersive estimate using stationary-phase method. In Section 4, we introduce some preliminary estimates. Combining basic energy estimates and space-time resonance, we conclude the bootstrap argument in Section 5.
\section{Proof of Theorem 1.1}
We plan to prove Theorem 1.1 using local existence theorem and bootstrap arguements. To begin with, we have
\begin{theorem} \textbf{Local well-posedness}\label{localthm}
	(1) For given $u_0\in {H^{4}(\mathbb{R})}$, there exists a constant $T_0>0$ determined by $ ||u_0||_{H^{4}(\mathbb{R})} $ and a unique solution
	$u \in C([-T_0,T_0], H^4)$ to Cauchy's problem \eqref{NFSE}.

	(2) Let $N\geq 4$, $u_0 \in H^N(\mathbb{R})$,  
	and $u \in C([-T_0,T_0], H^4)$ be a solution in (1). Hence, $u \in C([-T_0,T_0], H^N)$ and
	\begin{equation}\label{local}
		||u(t_2)||_{H^{N}}-||u(t_1)||_{H^{N}}\lesssim \int^{t_2}_{t_1}||u(s)||_{H^{N}}||u(s)||^2_{L^\infty}ds
	\end{equation}
	for any $t_1\leq t_2, t_1, t_2 \in [-T_0, T_0]$.
	
\end{theorem}
\begin{proof}
	Using a standard fixed-point theorem in space $$X=\left\{ v\in C[-T,T],H^4(\mathbb{R}):  \sup_{[-T_0,T_0]}||v(t)||_{H^4}\leq 3||u_0||_{H^4}\right\},$$
then there a unique solution
$u \in C([-T_0,T_0], H^4)$ to Cauchy's problem \eqref{NFSE}.
	
	The estimate \eqref{local} is from 
	\begin{equation*}
	u(t)=e^{it(-\Delta)^{-\frac{\alpha}{2}}}u_0-i\int^t_0 e^{i(t-s)(-\Delta)^{-\frac{\alpha}{2}}}c_*u(s)\bar{u}(s)u(s)ds.
	\end{equation*}
\end{proof}
\begin{proposition}\label{bootstrap1}
	Let $N_0=100, \ 0<\varepsilon_0<\varepsilon_1<\varepsilon_0^{\frac{2}{3}}<1$. Assume $u\in C([0,T],H^{N_0})$ $ (T>0) $ is a  solution for Cauchy's problem \eqref{NFSE}, and the initial data $u_0$
	satisfying
	\begin{equation*}
		||u_0||_{H^{N_0}(\mathbb{R})}+||x \cdot \partial u_0||_{L^{2}(\mathbb{R})}+||u_0||_{Z}\leq \varepsilon_0.
	\end{equation*}
	Let $f(t)=e^{\mathrm{i}t(-\Delta)^{\frac{\alpha}{2}}}u(t), t\in [0, T]$.


	Under assumptions of $p_0 \in (0, \frac{1}{1000}]$ and
	\begin{equation}\label{a1}
		\sup_{t \in [0,T]}\left\{(1+t)^{-p_0}||f(t)||_{H^{N_0}(\mathbb{R})}+(1+t)^{-p_0}||x\cdot{\partial f(t)}||_{L^{2}(\mathbb{R})}+||f(t)||_{Z} \right\} \leq \varepsilon_1,
	\end{equation}
 we could get estimates
	\begin{equation}
		\sup_{t \in[0,T] }\left\{(1+t)^{-p_0}||f(t)||_{H^{N_0}(\mathbb{R})}+(1+t)^{-p_0}||x\cdot{\partial f(t)}||_{L^{2}(\mathbb{R})}+||f(t)||_{Z}\right\} \lesssim \varepsilon_0,
	\end{equation}
	\qquad
and 
	\begin{equation}
		(1+t)^{p_1}
		\big| \big|
		(1+|\xi|)^{10} \left(e^{\textit{i}H(\xi,t)}
		\hat{f}(\xi,t)-e^{\textit{i}H(\xi,t_0)}\hat{f}(\xi,t_0) \right) \big| \big|_{L^{\infty}_\xi}
		\lesssim \varepsilon_0,\ \forall t \geq t_0, t, t_0\in[0,T].
	\end{equation}
	
\end{proposition}

\section{Modified dispersive estimates}

In this part, we will prove dispersive estimates using phase decomposition methods, which is a modified version compared with the dispersive estimate for  $\alpha=\frac{1}{2}$ in \cite{IP}.
\begin{Lemma}\label{disperlem}
	Let $\alpha \in (\frac{1}{3},1)$. For $\forall t \in \mathbb{R}^+$, we have the dispersive estimate
	\begin{equation}\label{disper-equ1}
		||e^{\mathrm{i}t(-\Delta)^{\frac{\alpha}{2}}}f||_{L^\infty}\lesssim (1+t)^{-\frac{1}{2}}|||\xi|^{\frac{2-\alpha}{2}}\hat{f}(\xi)||_{L^{\infty}_{\xi}} + (1+t)^{-(\frac{1}{2}+p_0)}(||f||_{H^2}+ ||x \cdot \partial f||_{L^2}).
	\end{equation}
\end{Lemma}
\begin{proof}
	Let $\varphi: \mathbb{R}\rightarrow [0,1]$ be an smooth function supported in $[-2,2]$ and $\varphi(x)\equiv 1$ for $x \in [-1,1]$. Denote
	\begin{equation*}
		\varphi_k(x):=\varphi(\frac{x}{2^k})-\varphi(\frac{x}{2^{k-1}}),\ k \in \mathbb{Z}, x \in \mathbb{R}.
	\end{equation*}
	In general, we define
	\begin{displaymath}
		\varphi_k^{(m)}(x) = \left\{ \begin{array}{ll}
			\varphi(\frac{x}{2^k})-\varphi(\frac{x}{2^{k-1}}) & \textrm{if $k\geq m+1$}\\
			\varphi(\frac{x}{2^k}) & \textrm{if $k=m$},
		\end{array} \right.
	\end{displaymath}
	where $m,k \in \mathbb{Z}, m\leq k$.
	To prove Lemma \ref{disperlem}, we need prove that the estimate 
	\begin{equation}\label{goal}
		||e^{\mathrm{i}t(-\Delta)^{\frac{\alpha}{2}}}f||_{L^\infty} \lesssim 1
	\end{equation}
hold under assumptions of
\begin{equation}\label{assumption1}
\big|\big||\xi|^{\frac{2-\alpha}{2}}\hat{f}(\xi)\big|\big|_{L^{\infty}_{\xi}} \lesssim (1+t)^{\frac{1}{2}}, 
\end{equation}
and
	\begin{equation}\label{assumption}
		||f||_{H^2}+ ||x \cdot \partial f||_{L^2} \lesssim (1+t)^{\frac{1}{2}+p_0}.
	\end{equation}
	Write
	\begin{align*}
		e^{\mathrm{i}t(-\Delta)^{\frac{\alpha}{2}}}f
		&= \mathcal{F}^{-1}(e^{\mathrm{i}t|\xi|^\alpha}\hat{f}(\xi))
		\\
		&=\int_{\mathbb{R}}e^{\mathrm{i}t|\xi|^\alpha}e^{\mathrm{i}x\xi}\hat{f}(\xi)d\xi
		\\
		&=\sum_{k\in \mathbb{Z}}\int_{\mathbb{R}}e^{\mathrm{i}t|\xi|^\alpha}e^{\mathrm{i}x\xi}\hat{f}(\xi)\varphi_k(\xi)d\xi.
	\end{align*}
	For low frequency part, it's not hard to get
	\begin{align*}
		&\sum_{2^k \leq 2^{10}(1+t)^{-(1+2p_0)}}\left|\int_{\mathbb{R}}e^{\mathrm{i}t|\xi|^\alpha}e^{\mathrm{i}x\xi}\hat{f}(\xi)\varphi_k(\xi)d\xi\right|
		\\
		=& \sum_{2^k \leq 2^{10}(1+t)^{-(1+2p_0)}}\left|\int_{\mathbb{R}}e^{\mathrm{i}t|\xi|^\alpha}e^{\mathrm{i}x\xi}\hat{P_k}f(\xi)d\xi\right|
		\\
		\lesssim & \sum_{2^k \leq 2^{10}(1+t)^{-(1+2p_0)}} 2^{\frac{k}{2}}
		||\hat{P_k}f(\xi)||_{L^2}
		\\
		\lesssim & (1+t)^{-(\frac{1}{2}+p_0)}\sum_{k\in \mathbb{Z}}||\hat{P_k}f(\xi)||_{L^2}
		\\
		\lesssim & 1,
	\end{align*}
	where $\hat{P_k}f(\xi)=\hat{f}(\xi)\varphi_k(\xi)$. The operator $P_k$ is also called the Littlewood-Paley operator, and we will use it frequently in the whole paper.

	For high frequency part, we have
	\begin{align*}
		&\sum_{2^k \geq 2^{-10}(1+t)}\big|\int_{\mathbb{R}}e^{\mathrm{i}t|\xi|^\alpha}e^{\mathrm{i}x\xi}\hat{f}(\xi)\varphi_k(\xi)d\xi\big|
		\\
		\lesssim & \sum_{2^k \geq 2^{-10}(1+t)} 2^{\frac{k}{2}}
		||\hat{P_k}f(\xi)||_{L^2}
		\\
		\lesssim & (1+t)^{-\frac{3}{2}}\sum_{k\in \mathbb{Z}}2^{2k}||\hat{P_k}f(\xi)||_{L^2}
		\\
		\lesssim & 1.
	\end{align*}
	Therefore, it remains to prove
	\begin{equation}\label{medfreq}
		\sum_{2^{10}(1+t)^{-(1+2p_0)}\leq 2^k \leq 2^{-10}(1+t)}\big|\int_{\mathbb{R}}e^{\mathrm{i}(t|\xi|^\alpha+x\xi)}\hat{f}(\xi)\varphi_k(\xi)d\xi\big|\lesssim 1.
	\end{equation}
	
	Let $\Phi(\xi)=t|\xi|^\alpha+x\xi$. Then we have
	\begin{equation*}
		\Phi'(\xi)=t(\frac{\alpha \xi}{|\xi|^{2-\alpha}}+\frac{x}{t}),
	\end{equation*}
	\begin{equation*}
		\Phi''(\xi)=\alpha(\alpha-1)t|\xi|^{-(2-\alpha)},
	\end{equation*}
	and
	\begin{equation*}
		\Phi'(\xi_0)=0 \Longrightarrow \xi_0=\alpha^{-\frac{1}{1-\alpha}}|\frac{t}{x}|^{\frac{1}{1-\alpha}}\sgn(\frac{t}{x}).
	\end{equation*}
	Assume that $t\geq 1$, from \eqref{assumption} we obtain
	\begin{equation*}
		||\hat{P_k}f(\xi)||_{L^2}+2^k||\partial\hat{P_k}f(\xi)||_{L^2}\lesssim t^{\frac{1}{2}+p_0}.
	\end{equation*}
	Therefore, if $|\frac{x}{t}|\leq 2^{-k(1-\alpha)-4}$ or $|\frac{x}{t}|\geq 2^{-k(1-\alpha)+4}$, we get 
	\begin{align*}
		&\sum_{k}\big| \int_{\mathbb{R}}e^{\mathrm{i}(t|\xi|^\alpha+x\xi)}\hat{P_k}f(\xi)d\xi \big|
		\\
		\lesssim &t^{-1}\sum_{k}2^{k(1-\alpha)}||\partial\hat{P_k}f(\xi)||_{L^1}+t^{-1}\sum_{k}2^{-k\alpha}||\hat{P_k}f(\xi)||_{L^1}
		\\
		:=&I+\varPi,
	\end{align*}
where $I=t^{-1}\sum_{k}2^{k(1-\alpha)}||\partial\hat{P_k}f(\xi)||_{L^1}, \varPi=t^{-1}\sum_{k}2^{-k\alpha}||\hat{P_k}f(\xi)||_{L^1}$.

In fact, we could estimate
	\begin{align*}
		I&=t^{-1}\sum_{k}2^{k(1-\alpha)}||\partial\hat{P_k}f(\xi)||_{L^1}
		\\
		&\lesssim t^{-1}\sum_{k\in \mathbb{Z}}2^{k(\frac{1}{2}-\alpha)}2^k||\partial\hat{P_k}f(\xi)||_{L^2},
	\end{align*}
	and
	\begin{align*}
		\varPi&=t^{-1}\sum_{k}2^{-k\alpha}||\hat{P_k}f(\xi)||_{L^1}
		\\
		&\lesssim t^{-1}\sum_{k\in \mathbb{Z}}2^{k(\frac{1}{2}-\alpha)}||\hat{P_k}f(\xi)||_{L^2}.
	\end{align*}
	\textit{Case 1}: $p_0 \leq \alpha \leq \frac{1}{2}$, we can get
	
	\begin{align*}
		I+\varPi&\lesssim t^{-1}t^{\frac{1}{2}-\alpha}\sum_{k\in \mathbb{Z}}(||\hat{P_k}f(\xi)||_{L^2}+2^k||\partial\hat{P_k}f(\xi)||_{L^2})
		\\
		&\lesssim t^{-\alpha+p_0}
		\\
		&\lesssim 1.
	\end{align*}
	\textit{Case 2}: $\frac{1}{2}<\alpha \leq \frac{1}{1+2p_0}$, we have
	\begin{align*}
		I+\varPi&\lesssim t^{-1}t^{(-\frac{1}{2}+\alpha)(1+2p_0)}\sum_{k\in \mathbb{Z}}(||\hat{P_k}f(\xi)||_{L^2}+2^k||\partial\hat{P_k}f(\xi)||_{L^2})
		\\
		&\lesssim t^{\alpha-1+2p_0\alpha}
		\\
		&\lesssim 1.
	\end{align*}
	Therefore, it suffices to prove
	\begin{equation}\label{left}
		\big|\int_{\mathbb{R}}e^{\mathrm{i}(t|\xi|^\alpha+x\xi)}\hat{f}(\xi)\varphi_k(\xi)d\xi\big|\lesssim 1
	\end{equation}
	provided that $t\geq 1$ and $2^k \in[2^{10}(1+t)^{-(1+2p_0)},2^{-10}(1+t)]$ $\bigcap  [2^{-\frac{4}{1-\alpha}}|\frac{t}{x}|^{\frac{1}{1-\alpha}}, 2^{\frac{4}{1-\alpha}}|\frac{t}{x}|^{\frac{1}{1-\alpha}}]$.
	In this situation, we notice that $|\xi_0|\approx 2^k$.

	Let $l_0$ denote the smallest integer with the property that $2^{2l_0}\geq t^{-1}2^{k(2-\alpha)}$ and we have 
	\begin{equation}\label{finalre}
		\big|\int_{\mathbb{R}}e^{\mathrm{i}(t|\xi|^\alpha+x\xi)}\hat{f}(\xi)\varphi_k(\xi)d\xi\big| \leq \sum_{l=l_0}^{k+100}|J_l(\xi)|,
	\end{equation}
	where
	\begin{equation*}
		J_l(\xi)=\int_{\mathbb{R}}e^{\mathrm{i}\Phi(\xi)}\hat{P_k}f(\xi) \varphi^{(l_0)}_l(\xi-\xi_0) d\xi.
	\end{equation*}
	Using \eqref{assumption}, we obtain
	\begin{equation*}
		||\hat{P_k}f(\xi)||_{L^\infty} \lesssim t^{\frac{1}{2}}2^{-k(1-\frac{\alpha}{2})}
	\end{equation*}
	and
	\begin{equation*}
		||\hat{P_k}f(\xi)||_{L^2}+2^k||\partial\hat{P_k}f(\xi)||_{L^2}\lesssim t^{\frac{1}{2}+p_0}.
	\end{equation*}
	Therefore,
	\begin{align*}
		J_{l_0}(\xi) & \lesssim 2^{l_0}||\hat{P_k}f(\xi)||_{L^\infty}
		\\ &\lesssim t^{-\frac{1}{2}}2^{k(1-\frac{\alpha}{2})}t^{\frac{1}{2}}2^{-k(1-\frac{\alpha}{2})}
		\\
		& \lesssim 1.
	\end{align*}
	Moreover, for $\Phi'(\xi)-\Phi'(\xi_0)=\Phi''(\eta)(\xi-\xi_0),  \eta \in (\xi_0,\xi)$, we get $|\Phi'(\xi)|\gtrsim |t|2^{-k(2-\alpha)}2^l$ whenever $|\xi|\approx 2^k$ and $|\xi-\xi_0|\approx 2^l$.

	For $l>l_0$, we integrate $J_l(\xi)$ by parts
	\begin{align*}
		J_l(\xi)&=\int_{\mathbb{R}}\frac{1}{\Phi'(\xi)}\hat{P_k}f(\xi) \varphi^{(l_0)}_l(\xi-\xi_0) de^{\mathrm{i}\Phi(\xi)}
		\\
		&=\int_{\mathbb{R}}\partial_{\xi}\frac{1}{\Phi'(\xi)}e^{\mathrm{i}\Phi(\xi)} \hat{P_k}f(\xi)\varphi^{(l_0)}_l(\xi-\xi_0) d\xi
		\\
		&\ \ \ \ +\int_{\mathbb{R}}\frac{1}{\Phi'(\xi)}e^{\mathrm{i}\Phi(\xi)}\partial_{\xi}\hat{P_k}f(\xi) \varphi^{(l_0)}_l(\xi-\xi_0) d\xi
		\\
		&\ \ \ \ +\int_{\mathbb{R}}\frac{1}{\Phi'(\xi)}e^{\mathrm{i}\Phi(\xi)} \hat{P_k}f(\xi)\partial_{\xi}\varphi^{(l_0)}_l(\xi-\xi_0) d\xi,
	\end{align*}
	and it follows that
	\begin{align*}
		|J_l(\xi)|&\lesssim \frac{1}{t2^{-k(2-\alpha)}2^l} \left(2^{-l}||\hat{P_k}f(\xi) 1_{[0,2^{l+4}]}(|\xi-\xi_0|)||_{L^1_{\xi}} +||\partial \hat{P_k}f(\xi) 1_{[0,2^{l+4}]}(|\xi-\xi_0|)||_{L^1_{\xi}} \right)
		\\
		& \lesssim t^{-1}2^{k(2-\alpha)}2^{-l}\left(||\hat{P_k}f(\xi)||_{L^\infty}+2^{\frac{l}{2}}||\partial\hat{P_k}f(\xi)||_{L^2}\right)
		\\
		& \lesssim t^{-\frac{1}{2}}2^{k(2-\alpha)-l-k(1-\frac{\alpha}{2})}+t^{-\frac{1}{2}}2^{k(1-\alpha)}2^{-\frac{l}{2}}.
	\end{align*}
	As a result, we have
	\begin{align*}
		\sum_{l=l_0+1}^{k+100}&\lesssim t^{-\frac{1}{2}}2^{-l_0}2^{k(1-\frac{\alpha}{2})}+
		t^{-\frac{1}{2}}2^{-\frac{l_0}{2}}2^{k(1-\alpha)}
		\\
		& \lesssim t^{-\frac{1}{2}}t^{\frac{1}{2}}2^{-k(1-\frac{\alpha}{2})}2^{k(1-\frac{\alpha}{2})}+
		t^{-\frac{1}{2}}2^{k(1-\alpha)}t^\frac{1}{4}2^{-\frac{k}{4}(2-\alpha)}
		\\
		&  \lesssim 1+t^{-\frac{1}{4}}2^{k(\frac{1}{2}-\frac{3}{4}\alpha)} := 1+ \Xi,
	\end{align*}
	where $\Xi:=t^{-\frac{1}{4}}2^{k(\frac{1}{2}-\frac{3}{4}\alpha)}$.
	Recalling $2^{10}(1+t)^{-(1+2p_0)} \leq 2^k \leq 2^{-10}(1+t)$, for $\frac{1}{3} < \alpha \leq \frac{1}{1+2p_0}$ we have 
	\begin{align*}
		\Xi&=t^{-\frac{1}{4}}2^{k(\frac{1}{2}-\frac{3}{4}\alpha)}
		\\
		& \lesssim \max \left\{ t^{\frac{1}{4}-\frac{3}{4}\alpha}, t^{-\frac{3}{4}(1-\alpha)+2p_0(\frac{3}{4}\alpha-\frac{1}{2})} \right\}
		\\
		& \lesssim 1.
	\end{align*}
\end{proof}
\section{Preliminaries}
In this section, we will introduce some basic estimates and commutator.

Let $P_j$ is the Littlewood-Paley operator and $\hat{f_j}=\hat{P_j} f$, $j \in \mathbb{Z}$. We denote
\begin{equation*}
I(\xi,t):=\int_{\mathbb{R}^2}e^{it\Psi(\xi,\eta,\sigma)}\hat{f}{}(\xi-\eta)\hat{f}(\eta-\sigma)\hat{\bar{f}}(\sigma)d\eta d\sigma,
\end{equation*}
where
\begin{equation*}
\Psi(\xi,\eta,\sigma):=|\xi|^\alpha-|\xi-\eta|^\alpha-|\eta-\sigma|^\alpha+|\sigma|^\alpha.
\end{equation*}
Using Littlewood-Paley decomposition, we have
$I(\xi,t)=\sum_{k_1,k_2,k_3 \in \mathbb{Z}}I_{k_1,k_2,k_3}(\xi,t)$.
Here
\begin{equation}\label{I_k}
I_{k_1,k_2,k_3}(\xi,t)=\int_{\mathbb{R}^2}e^{it\Psi(\xi,\eta,\sigma)}\hat{f}_{k_1}(\xi-\eta)\hat{f}_{k_2}(\eta-\sigma)\hat{\bar{f}}_{k_3}(\sigma)d\eta d\sigma.
\end{equation}
Without loss of generality, we assume $k_1 \leq k_2 \leq k_3$.

We also introduce the following estimates based on assumption \eqref{a1}, which will be used throughout our paper. In view of assumption \eqref{a1}, we could deduce that
\begin{equation}\label{a11}
||\hat{f}_l(s)||_{L^2}\lesssim \varepsilon_12^{p_0m}2^{-N_0l^+},
\end{equation}
\begin{equation}\label{a12}
||\partial \hat{f}_l(s)||_{L^2}\lesssim \varepsilon_12^{p_0m}2^{-l},
\end{equation}
\begin{equation}\label{a13}
||\hat{f}_l(s)||_{L^\infty}\lesssim \varepsilon_12^{-10 l^+} ,
\end{equation}
for any $l \in \mathbb{Z}$ and $s \in [2^m-2,2^{m+1}]  \bigcap[0,T]$.  
	\begin{Lemma}\label{m}
		\cite{IP}. Assume that $m \in L^1(\mathbb{R}^2)$ satisfies
		\begin{equation*}
			||\int_{\mathbb{R}^2}m(\eta,\sigma)e^{\textit{i}xy}e^{\textit{i}y\sigma}|| \leq A,
		\end{equation*}
		for some $A \in (0,\infty)$. Then for any $(p,q,r) \in \left\{(2,2,\infty), (2,\infty),(\infty,2,2) \right\}$, we have
		\begin{equation}
			\left| \int_{\mathbb{R}^2} \hat{f}(\eta)\hat{g}(\sigma)\hat{v}(-\eta-\sigma)m(\eta,\sigma)d\eta d\sigma \right| \lesssim A||f||_{L^p}||g||_{L^p}||v||_{L^p}.
		\end{equation}
	\end{Lemma}
	\begin{Lemma}\label{partialm}
	 Let $f(t)=e^{\mathrm{i}t(-\Delta)^{\frac{\alpha}{2}}}u(t)$ and $u$ is the unique local solution of \eqref{NFSE}. Under the assumption of Proposition \ref{bootstrap1}, for any $l \in \mathbb{Z}$ and $s \in [2^{m-1},2^{m+1}] \bigcap [0,T]$, we have
		\begin{equation}
			||\partial_s \hat{f}_l(s)||_{L^2} \lesssim \varepsilon_1 2^{3p_0m}2^{-20l^+}2^{-m},
		\end{equation}
		and
		\begin{equation}
			||\partial_s \hat{f}_l(s)||_{L^\infty} \lesssim \varepsilon_1 2^{3p_0m}2^{-20l^+}2^{-\frac{m}{2}}(2^{\frac{l}{2}}+2^{-\frac{m}{2}}).
		\end{equation}
	\end{Lemma}
\begin{proof}
	Firstly, we have $\partial_t \hat{f}(\xi,t)=-i(2\pi)^{-2}c_0I(\xi,t)$. We just need prove
\begin{equation}\label{e1}
|| \varphi_l(\xi)I(\xi,t)||_{L^2_\xi} \lesssim \varepsilon_1 2^{3p_0m}2^{-20l^+}2^{-m},
\end{equation}
and
\begin{equation}\label{e2}
|| \varphi_l(\xi)I(\xi,t)||_{L^\infty_\xi} \lesssim \varepsilon_1 2^{3p_0m}2^{-20l^+}2^{-\frac{m}{2}}\left(2^{\frac{l}{2}}+2^{-\frac{m}{2}}\right).
\end{equation}
Using \eqref{a11}, for $n \in \mathbb{Z}$, we have
\begin{equation}\label{440}
	|| e^{is (-\Delta)^{\frac{\alpha}{2}}} f_n(s) ||_{L^2_x} \lesssim \varepsilon_1 2^{p_0m}2^{-N_0n^+}2^{\frac{n}{2}}.
\end{equation}
Using dispersive estimate in Lemma \ref{disperlem}, we get
\begin{equation*}
	|| e^{is (-\Delta)^{\frac{\alpha}{2}}} f_n(\cdot,s) ||_{L^\infty} \lesssim \varepsilon_1 2^{-\frac{m}{2}}.
\end{equation*}
On the other hand, we have
\begin{align*}
	|| e^{is (-\Delta)^{\frac{\alpha}{2}}} f_n(\cdot,s) ||_{L^\infty} &\lesssim 
	|| \mathcal{F}^{-1}\left( e^{is (-\Delta)^{\frac{\alpha}{2}}} f_n(\cdot,s) \right) ||_{L^1_\xi}
	\\
	& \lesssim \varepsilon_12^{-10l^+} 2^{n}.
\end{align*}
Therefore
\begin{align}\label{441}
	|| e^{is (-\Delta)^{\frac{\alpha}{2}}} f_n(\cdot,s) ||_{L^\infty} 
	& \lesssim \varepsilon_1 \min\{2^{-\frac{m}{2}}, 2^n 2^{-10n^+}  \}.
\end{align}
To achieve our goal, we use phase decomposition
\begin{align*}
||\varphi_l(\xi)I(\xi,t)||_{L^2_\xi}&=||\varphi_l(\xi)\sum_{k_1,k_2,k_3 \in \mathbb{Z}}I_{k_1,k_2,k_3}(\xi,t)||_{L^2_\xi}
\\
& \lesssim \sum_{k_3 \geqslant l-10,k_2,k_1} \left| \int_{\mathbb{R}^2} e^{is\Psi(\xi,\eta,\sigma)}  \hat{f}_{k_1}(\xi-\eta,s)\hat{f}_{k_2}(\eta-\sigma,s)\hat{\bar{f}}_{k_3}(\eta,s) d\eta d\sigma\right|_{L^2_\xi}
\end{align*}
By passing to the physical space and estimating the highest frequency component in $L^2$ and the other two components in $L^\infty$, we get
\begin{align*}
&\sum_{k_3 \geqslant l-10,k_2,k_1} \left| \int_{\mathbb{R}^2} e^{is\Psi(\xi,\eta,\sigma)}  \hat{f}_{k_1}(\xi-\eta,s)\hat{f}_{k_2}(\eta-\sigma,s)\hat{\bar{f}}_{k_3}(\eta,s) d\eta d\sigma\right|_{L^2_\xi}
\\
& \lesssim \sum_{k_3 \geqslant l-10,k_2,k_1} \varepsilon_1 2^{p_0m}2^{-N_0k_3^+}2^{\frac{k_3}{2}} \cdot \varepsilon_1 \sqrt{2^{-\frac{m}{2}}2^{k_1}2^{-10k_1^+}}\cdot \varepsilon_1 \sqrt{2^{-\frac{m}{2}}2^{k_2}2^{-10k_2^+}}.
\\
& \lesssim \varepsilon_1 2^{-m}2^{3p_0m}2^{-20l^+}.
\end{align*}
Thus, we could conclude the proof of \eqref{e1}.

As for estimate \eqref{e2}, we will divide it into several cases.

\textit{Case 1}: For $l \geqslant 0$, by passing to the physical space and estimating the two highest frequency components in $L^2$ and the lowest frequency component in $L^\infty$, we get
\begin{align*}
&\sum_{k_3 \geqslant l-6,k_2,k_1} \left|\varphi_l(\xi) \int_{\mathbb{R}^2} e^{is\Psi(\xi,\eta,\sigma)}  \hat{f}_{k_1}(\xi-\eta,s)\hat{f}_{k_2}(\eta-\sigma,s)\hat{\bar{f}}_{k_3}(\eta,s) d\eta d\sigma\right|_{L^\infty_\xi}\\
\\
& \lesssim \sum_{k_3 \geqslant l-10,k_2,k_1} \varepsilon_1 2^{p_0m}2^{-N_0k_3^+}2^{\frac{k_3}{2}} \cdot \varepsilon_1 2^{p_0m}2^{-N_0k_2^+}2^{\frac{k_2}{2}} \cdot \varepsilon_12^{-\frac{m}{2}}
\\
& \lesssim \varepsilon_1 2^{-\frac{m}{2}}2^{3p_0m}2^{-10l^+}(2^{\frac{l}{2}}+2^{-\frac{m}{2}}).
\end{align*}
\quad \quad \textit{Case 2}: For $l < 0$. if $k_1 \leqslant \max\{ l,-m\}+10$ or $k_3 \geqslant \frac{m}{10}$, it's similar with \textit{Case 1}.

\textit{Case 3}: For $l < 0$ and $\max\{ l,-m\}+10 \leqslant k_1 ,k_2,k_3 \leqslant \frac{m}{10}$, therefore $|k_2-k_3| \leqslant 4$. Let $\chi$ be a smooth, cut-off function $\chi : \mathbb{R}\rightarrow [0,1]$ supported in $[-\frac{11}{10}, \frac{11}{10}]$ and equal to $1$ in $[-\frac{9}{10}, \frac{9}{10}]$. We define
\begin{equation*}
J_1=\int_{\mathbb{R}^2}\chi(\frac{\eta}{\sigma})e^{is\Psi(\xi,\eta,\sigma)}\hat{f}_{k_1}(\xi-\eta,s)\hat{f}_{k_2}(\eta-\sigma,s)\hat{\bar{f}}_{k_3}(\eta,s)d\eta d\sigma,
\end{equation*}
\begin{equation*}
J_1=\int_{\mathbb{R}^2} \left(1-\chi(\frac{\eta}{\sigma}) \right)e^{is\Psi(\xi,\eta,\sigma)}\hat{f}_{k_1}(\xi-\eta,s)\hat{f}_{k_2}(\eta-\sigma,s)\hat{\bar{f}}_{k_3}(\eta,s)d\eta d\sigma.
\end{equation*}
To estimate $J_1$ we integrate by parts in $\sigma$. Recall the definition of $\Psi(\xi,\eta,\sigma)$, which shows that
\begin{equation*}
\left|\partial_\sigma \Psi(\xi,\eta,\sigma) \right| \lesssim 2^{(\alpha-1) k_2}. 
\end{equation*}
Integrating $J_1$ by parts in $\sigma$, we get
\begin{align*}
	|J_1| \lesssim J_{11}+J_{12}+J_{13},
\end{align*}
where
$$J_{11}:=\int_{\mathbb{R}^2} 
\frac{1}{2^m 2^{(\alpha-1) k_2}} \left| \hat{f}_{k_1}(\xi-\eta,s)\right| 2^{-k_2} \left| \hat{f}_{k_2}(\eta-\sigma,s) \hat{\bar{f}}_{k_3}(\sigma,s)\right|d\eta d\sigma,$$
$$J_{12}:=\int_{\mathbb{R}^2} 
\frac{1}{2^m 2^{(\alpha-1) k_2}} \left| \hat{f}_{k_1}(\xi-\eta,s)\right| \left| \partial\hat{f}_{k_2}(\eta-\sigma,s)\right| \left| \hat{\bar{f}}_{k_3}(\sigma,s)\right|d\eta d\sigma,$$
$$J_{13}:=\int_{\mathbb{R}^2} 
\frac{1}{2^m 2^{(\alpha-1) k_2}} \left| \hat{f}_{k_1}(\xi-\eta,s)\right|  \left| \hat{f}_{k_2}(\eta-\sigma,s)\right| \left| \partial \hat{\bar{f}}_{k_3}(\sigma,s)\right|d\eta d\sigma.$$
Using \eqref{a11},\eqref{a12},\eqref{a13}, we deduce that
\begin{align*}
J_{11} &\lesssim 2^{-m}2^{(1-\alpha)k_2}2^{-k_2} \varepsilon_1^3 \min \big\{2^{-10(k_1^+ + k_2^+ + k_3^+)}2^{k_1+k_2+k_3} \cdot 2^{k_2+k_3},
 \\
& \qquad \qquad \qquad \qquad \qquad \qquad \quad 2^{-10k_1^+}2^{k_1} \cdot2^{p_0m}2^{-N_0(k_2^++k_3^+)}2^{\frac{k_2+k_3}{2}} \big\}\\
&\lesssim \varepsilon_12^{-m}2^{2p_0m}
\\
&\lesssim \varepsilon_12^{-m}2^{2p_0m}2^{-20l^+}\left( 2^{\frac{l}{2}}+2^{-\frac{m}{2}}\right),
\end{align*}
\begin{align*}
J_{12} &\lesssim 2^{-m}2^{(1-\alpha)k_2} \varepsilon_1^3\min \big\{2^{-10(k_1^+  + k_3^+)}2^{k_1+k_3}\cdot 2^{p_0m}2^{-k_2} \cdot 2^{k_2+k_3}, 
\\
&\qquad \qquad \qquad \qquad \qquad \quad 2^{p_0m}2^{-N_0(k_1^++k_3^+)}2^{\frac{k_1+k_3}{2}}\cdot 2^{p_0m}2^{-k_2} \big\}\\
&\lesssim \varepsilon_12^{-m}2^{2p_0m}
\\
&\lesssim \varepsilon_12^{-m}2^{2p_0m}2^{-20l^+}\left( 2^{\frac{l}{2}}+2^{-\frac{m}{2}}\right),
\end{align*}
\begin{align*}
J_{12} &\lesssim 2^{-m}2^{(1-\alpha)k_2} \varepsilon_1^3\min \big\{2^{-10(k_1^+  + k_2^+)}2^{k_1+k_3}\cdot 2^{p_0m}2^{-k_3} \cdot 2^{k_2+k_3},
\\
&\qquad \qquad \qquad \qquad \qquad \quad  2^{p_0m}2^{-N_0(k_1^++k_2^+)}2^{\frac{k_1+k_2}{2}}\cdot 2^{p_0m}2^{-k_3} \big\}\\
&\lesssim \varepsilon_12^{-m}2^{2p_0m}
\\
&\lesssim \varepsilon_12^{-m}2^{2p_0m}2^{-20l^+}\left( 2^{\frac{l}{2}} + 2^{-\frac{m}{2}}\right),
\end{align*}
Thus, we finish the proof of this Lemma.
\end{proof}
\begin{Lemma}
	Let $0<\alpha \leq 2$ and $S=\alpha t \partial_t+x\partial_x$ be the scaling vector field and we have
	\begin{equation*}
		[S, \mathrm{i} \partial_t-(-\Delta)^{\frac{\alpha}{2}}]=-\alpha \mathrm{i}(\mathrm{i} \partial_t-(-\Delta)^{\frac{\alpha}{2}}).
	\end{equation*}
	Here $[A,B]:=AB-BA$.
\end{Lemma}
\begin{proof}
	It is easy to check
	\begin{equation*}
		[S,\mathrm{i} \partial_t]=-\alpha \mathrm{i} \partial_t,
	\end{equation*}
	and for any smooth $g$ with compact support,
	\begin{align*}
		([x\partial_x, (-\Delta)^{\frac{\alpha}{2}}]g)^{\hat{}}
		&=(x\partial_x(-\Delta)^{\frac{\alpha}{2}}g)^{\hat{}}-((-\Delta)^{\frac{\alpha}{2}}x\partial_xg)^{\hat{}}
		\\
		&=\partial_\xi(\xi|\xi|^\alpha \hat{g})-|\xi|^\alpha \partial_\xi(\xi \hat{g})
		\\
		&=\big(\partial_\xi(\xi|\xi|^\alpha )-|\xi|^\alpha \big)\hat{g}
		\\
		&=\alpha |\xi|^\alpha \hat{g}.
	\end{align*}
	As a result, we get $[S, \mathrm{i} \partial_t-(-\Delta)^{\frac{\alpha}{2}}]=-\alpha \mathrm{i}(\mathrm{i} \partial_t-(-\Delta)^{\frac{\alpha}{2}})$.	
\end{proof}
\section{Proof of Propositon 2.1}
In this section, we will give the proof of Propositon 2.1 combining basic energy estimates and space-time resonance. 
\begin{Lemma}\label{yuxiang}
	For any $t \in [0,T]$,
	we have
	\begin{equation}\label{40}
		||S(|u|^2u)(t)||_{L^{2}(\mathbb{R})}\leq 4\varepsilon_0(1+t)^{-1+p_0},
	\end{equation}
	where $S=\alpha t \partial_t+x\partial_x$.
\end{Lemma}
\begin{proof}
	If $t=0$, we have
	\begin{align*}
		||S(|u|^2u)(0)||_{L^{2}(\mathbb{R})} & \leq
		||u_0||^2_{L^{\infty}(\mathbb{R})}||x \cdot \partial u_0||_{L^{2}(\mathbb{R})}
		\\
		& \leq 2\varepsilon_0.
	\end{align*}
	For any $t \in (0,T]$, we assume that
	\begin{equation}
		||S(|u|^2u)(t)||_{L^{2}(\mathbb{R})}\leq 4\varepsilon_0(1+t)^{-1+p_0}.
	\end{equation}
	By the assumption \eqref{a1} and Lemma \ref{disperlem}, it follows easily that
	\begin{align*}
		||S(|u|^2u)(t)||_{L^{2}(\mathbb{R})} & \leq ||u||^2_{L^{\infty}(\mathbb{R})}||Su(t)||_{L^{2}(\mathbb{R})}
		\\
		&	\lesssim \varepsilon^2_1 (1+t)^{-1}\varepsilon_0(1+t)^{-1+p_0}
		\\
		&	\lesssim \varepsilon_0 \varepsilon^2_1 (1+t)^{-2+p_0}
		\\
		& \leq 2\varepsilon_0(1+t)^{-1+p_0}.
	\end{align*}
	Using the bootstrap arguement, we conclude the proof of \eqref{40}.
\end{proof}
\begin{Lemma}
	With the same assumptions in Proposition \ref{bootstrap1}, we have, for any $t \in [0,T]$,
	\begin{equation}
		||u(t)||_{H^{N_0}(\mathbb{R})}+||Su(t)||_{L^{2}(\mathbb{R})} \lesssim \varepsilon_0(1+t)^{p_0}.
	\end{equation}
\end{Lemma}
\begin{proof}
	Following basic energy estimate \eqref{local} and  dispersive estimate \eqref{disper-equ1}, for any $t \in[0,T]$, we have
	\begin{align*}
		||u(t)||_{H^{N_0}(\mathbb{R})} & \lesssim ||u_0||_{H^{N_0}(\mathbb{R})}+\int^{t}_{0}||u(s)||_{H^{N_0}}||u(s)||^2_{L^\infty}ds \\
		& \lesssim \varepsilon_0+\int^{t}_{0}\varepsilon_1(1+s)^{p_0} \varepsilon^2_1(1+s)^{-1}ds
		\\
		& \lesssim \varepsilon_0+{\varepsilon^3_1}(1+t)^{p_0}
		\\
		& \lesssim \varepsilon_0(1+t)^{p_0}.
	\end{align*}
	Note 
	\begin{align*}
		(\mathrm{i} \partial_t-(-\Delta)^{\frac{\alpha}{2}})Su&=[S, \mathrm{i} \partial_t-(-\Delta)^{\frac{\alpha}{2}}]u-S(\mathrm{i} \partial_t u-(-\Delta)^{\frac{\alpha}{2}}u)\\
		&=\mathrm{i} \alpha |u^2|u-S(|u|^2u).
	\end{align*}
	Therefore, we have
	\begin{align*}
		||Su(t)||_{L^{2}(\mathbb{R})} &\lesssim ||Su(0)||_{L^{2}(\mathbb{R})}+\int^t_0 |||u^2|u(s)||_{L^{2}(\mathbb{R})}ds+\int^t_0 ||S(|u|^2u)(s)||_{L^{2}(\mathbb{R})}ds
		\\ &\lesssim ||x \cdot \partial u_0||_{L^{2}(\mathbb{R})}+\int^t_0 \varepsilon^3_1(1+s)^{-1+p_0}ds+\int^t_0 \varepsilon_0(1+s)^{-1+p_0}ds
		\\
		& \lesssim \varepsilon_0(1+t)^{p_0}.
	\end{align*}
\end{proof}
\begin{Lemma}
	With the same assumptions in Proposition \ref{bootstrap1}, we have, for any $t \in [0,T]$,
	\begin{equation}\label{390}
		\sup_{t\geq 0}||x\cdot{\partial f(t)}||_{L^{2}(\mathbb{R})} \lesssim \varepsilon_0(1+t)^{p_0}.
	\end{equation}
\end{Lemma}
\begin{proof}
	Let $f(t)=e^{\mathrm{i}t(-\Delta)^{\frac{\alpha}{2}}}u(t), t\in [0, T]$ and we have
	\begin{equation*}
		e^{\mathrm{i}t(-\Delta)^{\frac{\alpha}{2}}}\partial_t f=(\partial_t-\mathrm{i}(-\Delta)^{\frac{\alpha}{2}})u,
	\end{equation*}
	\begin{equation*}
		Su=e^{\mathrm{i}t(-\Delta)^{\frac{\alpha}{2}}}(x \cdot \partial f)+\alpha t e^{\mathrm{i}t(-\Delta)^{\frac{\alpha}{2}}}(\partial_t f).
	\end{equation*}
	Therefore, it follows
	\begin{align*}
		||x\cdot{\partial f(t)}||_{L^{2}(\mathbb{R})}
		&\lesssim ||Su(t)||_{L^{2}(\mathbb{R})}+
		t||\partial_t f||_{L^{2}(\mathbb{R})}
		\\
		&\lesssim ||Su(t)||_{L^{2}(\mathbb{R})}+
		t||e^{-\mathrm{i}t(-\Delta)^{\frac{\alpha}{2}}}(u^2\bar{u})||_{L^{2}(\mathbb{R})}
		\\
		&\lesssim ||Su(t)||_{L^{2}(\mathbb{R})}+
		t||u^2\bar{u}||_{L^{2}(\mathbb{R})}
		\\
		&\lesssim \varepsilon_0(1+t)^{p_0}+\varepsilon_0 t(1+t)^{-1+p_0}
		\\
		&\lesssim \varepsilon_0(1+t)^{p_0}.
	\end{align*}
\end{proof}
\begin{Lemma}\label{lemkey}
	With the same assumptions in Proposition \ref{bootstrap1}, we have, for any $t \in [0,T]$,
	\begin{equation}\label{Zgoal}
		\sup_{t\geq 0}||f(t)||_{Z} \lesssim \varepsilon_0.
	\end{equation}
\end{Lemma}
\begin{proof}
	For any $t \in [2^m-2,2^{m+1}] \cap [0,T]$ with $m\in \mathbb{Z}^+$,  $|\xi|\in [2^k,2^{k+1}]$ with $k \in \mathbb{Z}$ and
	\begin{equation*}
		k \in (-\infty, -5p_0m)\cup(5p_0m, +\infty),
	\end{equation*}
	we obtain
	\begin{align*}
		\left(1+|\xi|\right)^{10}||\hat{P}_kf(\xi,t)||_{L^\infty_\xi}
		& \leq (1+2^k)^{10}\{2^{-k}||\hat{P}_kf||_{L^2}(2^k)||\partial \hat{P}_kf||_{L^2}+||\hat{P}_kf||_{L^2}\}^{\frac{1}{2}}
		\\
		& \lesssim \varepsilon_0 (1+t)^{p_0}2^{-\frac{|k|}{2}}.
	\end{align*}
	As a result,
	\begin{align*}
		\sum_{|k|>5p_0m}\left(1+|\xi|\right)^{10}||\hat{P}_kf(\xi,t)||_{L^\infty_\xi} &\lesssim \varepsilon_0 (1+t)^{p_0}\sum_{|k|>5p_0m}2^{-\frac{|k|}{2}}
		\\
		& \lesssim \varepsilon_0 (1+t)^{-\frac{3}{2}p_0}
		\\
		& \lesssim \varepsilon_0.
	\end{align*}
	It remains to prove for $k \in [-5p_0m,5p_0m]$
	\begin{equation}\label{41}
		\left(1+|\xi|\right)^{10}|\hat{P}_kf(\xi,t)| \lesssim \varepsilon_0,
	\end{equation}
	where
	\begin{equation*}
		\hat{f}(\xi,t)=\hat{u}_0(\xi)-\textit{i}(2\pi)^{-2}c_*\int_{\mathbb{R}^2}e^{it\Psi(\xi,\eta,\sigma)}\hat{f}{}(\xi-\eta)\hat{f}(\eta-\sigma)\hat{\bar{f}}(\sigma)d\eta d\sigma.
	\end{equation*}
	Note that the contribution of space-time resonances lie in the points where
	\begin{equation*}
		\Psi(\xi,\eta,\sigma)=\partial_\eta \Psi(\xi,\eta,\sigma)=\partial_\sigma \Psi(\xi,\eta,\sigma)=0.
	\end{equation*}
	Therefore, the spacetime resonance arises in $(\xi,\eta,\sigma) \in \{(\xi,0,\xi),(\xi,0,-\xi)\}$. Fortunately,  these points are not absolutely integrable in time. This is to a large extent the key of establishing \eqref{41} for $k \in [-5p_0m,5p_0m]$.

	To achieve this goal, we recall
	\begin{equation*}
		H(\xi,t)=c_0c_*|\xi|^{2-\alpha}\int^t_0 \left|\hat{f}(\xi,r)\right|^2\frac{dr}{r+1}, \ c_0=\frac{2\pi}{\alpha(1-\alpha)},
	\end{equation*}
	and we have
	\begin{equation*}
		\frac{\textit{d}}{\textit{d}t} \left\{\hat{f}(\xi,t)e^{\textit{i}H(\xi,t)}\right\}=-\textit{i}(2\pi)^{-2}c_*e^{\textit{i}H(\xi,t)}\left( I(\xi,t)-c_0|\xi|^{2-\alpha}\frac{|\hat{f}(\xi,t)|^2\hat{f}(\xi,t)}{t+1} \right) ,
	\end{equation*}
	where
	\begin{equation*}
		I(\xi,t):=\int_{\mathbb{R}^2}e^{it\Psi(\xi,\eta,\sigma)}\hat{f}{}(\xi-\eta)\hat{f}(\eta-\sigma)\hat{\bar{f}}(\sigma)d\eta d\sigma.
	\end{equation*}
	By the decomposition of frequency, one can reduce \eqref{41} to Lemma \ref{reduction}. Thus, we conclude our proof of Lemma \ref{lemkey} by Lemma \ref{reduction}.
\end{proof}

\begin{Lemma}\label{reduction}
	Assume that $k \in [-5p_0m,5p_0m]$, $m\in \{1,2,\cdots \}$, $\xi \in [2^k,2^{k+1}]$, and $t_1 \leq t_2 \in [2^{m}-2,2^{m+1}] \cap [0,T]$. Then
	\begin{align}\label{final}
		&\sum_{k_1,k_2,k_3 \in \mathbb{Z}}\left|\int^{t_2}_{t_1}e^{iH(\xi,s)}\left[I_{k_1,k_2,k_3}(\xi,s)-c_0c_*\frac{|\xi|^{2-\alpha}\hat{f}_{k_1}(\xi,s)\hat{f}_{k_2}(\xi,s)\hat{\bar{f}}_{k_3}(-\xi,s)}{s+1}\right]ds \right|
		\nonumber
		\\
		&\leq \varepsilon^3_12^{-p_0m}2^{-10k^+},
	\end{align}
	where $H(\xi,s)=c_0c_*|\xi|^{2-\alpha}\int^s_0 \left|\hat{f}(\xi,r)\right|^2\frac{dr}{r+1}$ and $c_0=\frac{2\pi}{\alpha(1-\alpha)}$.
\end{Lemma}
We will prove Lemma \ref{reduction} in different cases. 
\begin{Lemma}\label{46}
	The bound \eqref{final} holds provided that
	\begin{align}
		k_3 \geq \frac{1+53p_0}{7}m \ \ \ or \ \ \ k_1 \leq -4m.
	\end{align}
\end{Lemma}
\begin{proof}
	Using the fact that $k \in [-5p_0m,5p_0m]$, we have when $k_3 \geq \frac{1+53p_0}{7}m$
	\begin{align*}
		||I_{k_1,k_2,k_3}(s)||_{L^\infty}
		& \leq ||\hat{f}_{k_1}(s)||_{L^\infty}||\hat{f}_{k_2}(s)||_{L^2}
		||\hat{f}_{k_3}(s)||_{L^2}
		\\
		&\lesssim
		\varepsilon_1^3 2^{-10k_3^+}2^{2p_0m}2^{-N_0(k_1^+ + k^+_2)}
		\\
		& \lesssim \varepsilon^3_1 2^{-m}2^{-p_0m}2^{-10k^+}.
	\end{align*}
	If $k_1 \leq -4m$, we get
	\begin{align*}
		||I_{k_1,k_2,k_3}(s)||_{L^\infty}
		& \leq 2^{\frac{k_1}{2}}||\hat{f}_{k_3}(s)||_{L^2}||\hat{f}_{k_1}(s)||_{L^\infty}||\hat{f}_{k_2}(s)||_{L^\infty}
		\\
		&\lesssim
		 2^{-2m}\cdot \varepsilon_12^{p_0m}2^{-N_0k_3^+}\cdot \varepsilon_12^{-10k_1^+}\cdot\varepsilon_12^{-10k_2^+}
		\\
		& \lesssim \varepsilon^3_1 2^{-m}2^{-p_0m}2^{-10k^+},
	\end{align*}
	where we use \eqref{a11} and \eqref{a13}.

	Moreover, using (4.9), we have
	\begin{align*}
	&	\left|\frac{|\xi|^{2-\alpha}\hat{f}_{k_1}(\xi,s)\hat{f}_{k_2}(\xi,s)\hat{\bar{f}}_{k_3}(-\xi,s)}{s+1} \right|
	\\
		& \lesssim 2^{k(2-\alpha)}2^{-m}\left[ \varepsilon_1 (1+2^{10 k})\right]^{-3} \cdot \mathbf{1}_{[0,4]}\left( \max(|k_1-k|, |k_1-k|,|k_1-k|)\right)
		\\
		&\lesssim \varepsilon_1^32^{-m}2^{-p_0m}2^{-10k^+}.
	\end{align*}
\end{proof}
\begin{Lemma}
	The bound \eqref{final} holds provided that
	\begin{align}
		& |k| \leq 5p_0m,
		\\
		\nonumber
		&k_1,k_2,k_3 \in [-4m, \frac{53p_0+1}{7}m],
		\nonumber
		\\
		& \max \{|k_1-k|,|k_2-k|,|k_3-k|\} \leq 16.
	\end{align}
\end{Lemma}
\begin{proof}
	Indeed we will prove for any $s \in[t_1,t_2]$,
	\begin{equation}\label{Iprrof}
		\left|I_{k_1,k_2,k_3}(\xi,s)-c_0\frac{|\xi|^{2-\alpha}\hat{f}_{k_1}(\xi,s)\hat{f}_{k_2}(\xi,s)\hat{\bar{f}}_{k_3}(-\xi,s)}{s+1} \right| \lesssim \varepsilon^3_12^{-m}2^{-p_0m}2^{-10k_+},
	\end{equation}
	which implies the desired bound. Changing the variables gives
	\begin{equation*}
		I_{k_1,k_2,k_3}(\xi,s)=\int_{\mathbb{R} \times \mathbb{R}}e^{is\Phi(\xi,\eta,\sigma)}\hat{f}_{k_1}(\xi+\eta,s)\hat{f}_{k_2}(\xi+\sigma,s)\hat{\bar{f}}_{k_3}(-\xi-\eta-\sigma,s)d\eta d\sigma,
	\end{equation*}
	where $\Phi(\xi,\eta,\sigma)=|\xi|^\alpha-|\xi+\eta|^\alpha-|\xi+\sigma|^\alpha+|\xi+\eta+\sigma|^\alpha$.
	Let $\bar{l}$ denote the smallest integer with the property that
	\begin{equation}\label{lbar}
		2^{2\bar{l}} \geq 2^{k(2-\alpha)}s^{-1}\geq 2^{k(2-\alpha)}2^{-(m+1)}.
	\end{equation}

	And we have
	\begin{equation}\label{decom}
		\left|I_{k_1,k_2,k_3}(\xi,s)\right|\leq \sum^{k+16}_{l_1,l_2=\bar{l}} \left|J_{l_1,l_2}(\xi,s)\right|,
	\end{equation}
	where
	\begin{equation}\label{Jl1l2}
		J_{l_1,l_2}(\xi,s)=\int_{\mathbb{R} \times \mathbb{R}}e^{is\Phi(\xi,\eta,\sigma)}\hat{f}_{k_1}(\xi+\eta,s)\hat{f}_{k_2}(\xi+\sigma,s)\hat{\bar{f}}_{k_3}(-\xi-\eta-\sigma,s)\varphi^{(\bar{l})}_{l_1}(\eta)\varphi^{(\bar{l})}_{l_2}(\sigma)d\eta d\sigma
	\end{equation}
	\textit{Case 1}: $l_1 \geq \max\{l_2,\bar{l}+1\}$.

	It is sufficient to show
	\begin{equation*}
		\left|J_{l_1,l_2}(\xi,s)\right| \lesssim 2^{-m}\varepsilon_1^3 2^{-3p_0m}2^{-10k^+}.
	\end{equation*}

	To achieve this goal, we will integrate by parts with respect to $\eta$ in the formula \eqref{Jl1l2}. Note that
	\begin{align}\label{region}
		\left|\partial_\eta \Phi(\xi,\eta,\sigma)\right|&=\alpha \left| sgn(\xi+\eta+\sigma)\right| \left|\xi+\eta+\sigma \right|^{\alpha-1}-|\xi+\eta|^{\alpha-1}\left|sgn(\xi+\eta)\right|
		\\
		& \gtrsim 2^{l_2}2^{k(\alpha-2)},
	\end{align}
	where $|\xi+\eta|\approx 2^k$, $|\xi+\eta+\sigma| \approx 2^k$, $|\sigma| \approx 2^{l_2}$.

	Then we have
	\begin{equation*}
		\left|J_{l_1,l_2}(\xi,s) \right| \leq \left|J_{l_1,l_2,1}(\xi,s)\right|+\left|F_{l_1,l_2,1}(\xi,s)\right|+\left|G_{l_1,l_2,1}(\xi,s)\right|,
	\end{equation*}
	where
	\begin{align}
		&J_{l_1,l_2,1}(\xi,s)=\int_{\mathbb{R}^2}e^{is\Phi(\xi,\eta,\sigma)}\hat{f}_{k_1}(\xi+\eta,s)\hat{f}_{k_2}(\xi+\sigma,s)\hat{\bar{f}}_{k_3}(-\xi-\eta-\sigma,s)
		(\partial_\eta m_1)(\eta,\sigma)d\eta d\sigma
		\nonumber
		\\
		& F_{l_1,l_2,1}(\xi,s)=\int_{\mathbb{R}^2}e^{is\Phi(\xi,\eta,\sigma)}\partial\hat{f}_{k_1}(\xi+\eta,s)\hat{f}_{k_2}(\xi+\sigma,s)\hat{\bar{f}}_{k_3}(-\xi-\eta-\sigma,s) m_1(\eta,\sigma)d\eta d\sigma
		\\
		& G_{l_1,l_2,1}(\xi,s)=\int_{\mathbb{R}^2}e^{is\Phi(\xi,\eta,\sigma)}\hat{f}_{k_1}(\xi+\eta,s)\hat{f}_{k_2}(\xi+\sigma,s)\partial\hat{\bar{f}}_{k_3}(-\xi-\eta-\sigma,s) m_1(\eta,\sigma)d\eta d\sigma
		\nonumber
	\end{align}
	and
	\begin{equation*}
		m_1(\eta,\sigma)=\frac{\varphi^{(\bar{l})}_{l_1}(\eta)\varphi_{l_2}(\sigma)}{s\partial_\eta\Phi(\xi,\eta,\sigma)}\varphi_{[k_1-2,k_1+2]}(\xi+\eta)\varphi_{[k_3-2,k_3+2]}(\xi+\eta+\sigma).
	\end{equation*}

	To estimate $F_{l_1,l_2,1}(\xi,s)$, we first write
	\begin{equation*}
		\hat{b}(\eta)=e^{-is|\xi+\eta|^\alpha}\partial\hat{f}_{k_1}(\xi+\eta,s),
	\end{equation*}
	\begin{equation*}
		\hat{g}(\sigma)=e^{-is|\xi+\sigma|^\alpha}\hat{f}_{k_2}(\xi+\sigma,s) \varphi(\frac{\sigma}{2^{l_2+4}}),
	\end{equation*}
	\begin{equation*}
		\hat{h}(-\eta-\sigma)=e^{is|\xi+\eta+\sigma|^\alpha}\hat{\bar{f}}_{k_3}(-\xi-\eta-\sigma,s) \varphi(\frac{-\eta-\sigma}{2^{l_2+4}}).
	\end{equation*}
	
	Use \eqref{a11}, \eqref{440},\eqref{441}, we have
	\begin{align*}
		&||b||_{L^2} \lesssim \varepsilon_12^{-k}2^{p_0m},
		\\
		& ||g||_{L^\infty} \lesssim \varepsilon_1 2^{-\frac{m}{2}},
		\\
		& ||h||_{L^2} \lesssim \varepsilon_12^{\frac{l_2}{2}}2^{-10k^+}.
	\end{align*}

	It is easy to verify, comparing with \eqref{region}, that $m_1$ satisfies the stronger estimate
	\begin{equation}\label{region1}
		\left|\partial_\eta m_1(\eta,\sigma)\right|\lesssim (2^{-m}2^{-l_2}2^{k(2-\alpha)})2^{-l_1}1_{[0,2^{l_1+4}]}(|\eta|)1_{[2^{l_2-4},2^{l_2+4}]}(|\sigma|).
	\end{equation}
	Therefore, we get
	\begin{equation*}
		||\mathcal{F}^{-1}m_1||_{L^1} \lesssim 2^{-m}2^{-l_2}2^{k(2-\alpha)}.
	\end{equation*}

	Note that since $2^{-\frac{l_2}{2}} \lesssim 2^{k(\frac{1}{2}-\frac{\alpha}{4})}2^{-\frac{m}{4}}$ and $|k|\leq 5p_0m$, we have
	\begin{align*}
		\left|F_{l_1,l_2,1}(\xi,s)\right|
		&\lesssim ||b||_{L^2}||g||_{L^\infty}||h||_{L^2}||\mathcal{F}^{-1}m_1||_{L^1}
		\\
		&\lesssim \varepsilon_1 2^{-k}2^{p_0m}\cdot \varepsilon_12^{-\frac{m}{2}}\cdot \varepsilon_12^{\frac{l_2}{2}}2^{-10k^+} \cdot 2^{-m}2^{-l_2}2^{k(2-\alpha)}
		\\
		& \lesssim \varepsilon^3_1 2^{-m} 2^{-p_0m}2^{-10k^+}.
	\end{align*}
	Similarly, the following holds
	\begin{equation*}
		\left|G_{l_1,l_2,1}(\xi,s)\right| \lesssim \varepsilon^3_1 2^{-m} 2^{-p_0m}2^{-10k^+}.
	\end{equation*}

	As for $J_{l_1,l_2,1}(\xi,s)$,  integrate by parts again with respect to $\eta$,
	\begin{equation*}
		\left|J_{l_1,l_2,1}(\xi,s)\right| \leq \left|J_{l_1,l_2,2}(\xi,s)\right|+\left|F_{l_1,l_2,2}(\xi,s)\right|+\left|G_{l_1,l_2,2}(\xi,s)\right|,
	\end{equation*}
	where
	\begin{align}
		&J_{l_1,l_2,2}(\xi,s)=\int_{\mathbb{R}^2}e^{is\Phi(\xi,\eta,\sigma)}\hat{f}_{k_1}(\xi+\eta,s)\hat{f}_{k_2}(\xi+\sigma,s)\hat{\bar{f}}_{k_3}(-\xi-\eta-\sigma,s)
		(\partial_\eta m_2)(\eta,\sigma)d\eta d\sigma
		\nonumber
		\\
		& F_{l_1,l_2,2}(\xi,s)=\int_{\mathbb{R}^2}e^{is\Phi(\xi,\eta,\sigma)}\partial\hat{f}_{k_1}(\xi+\eta,s)\hat{f}_{k_2}(\xi+\sigma,s)\hat{\bar{f}}_{k_3}(-\xi-\eta-\sigma,s) m_2(\eta,\sigma)d\eta d\sigma
		\\
		& G_{l_1,l_2,2}(\xi,s)=\int_{\mathbb{R}^2}e^{is\Phi(\xi,\eta,\sigma)}\hat{f}_{k_1}(\xi+\eta,s)\hat{f}_{k_2}(\xi+\sigma,s)\partial\hat{\bar{f}}_{k_3}(-\xi-\eta-\sigma,s) m_2(\eta,\sigma)d\eta d\sigma
		\nonumber
	\end{align}
	and
	\begin{equation*}
		m_2(\eta,\sigma)=\frac{(\partial_\eta m_1)(\eta,\sigma)}{s(\partial_\eta \Phi)(\eta,\sigma)}.
	\end{equation*}

	Compared with $m_1$, $m_2$ satisfies the stronger bounds
	\begin{equation}\label{region3}
		\left|m_2(\eta,\sigma)\right|\lesssim (2^{-m}2^{-l_1-l_2}2^{k(2-\alpha)})(2^{-m}2^{-l_2}2^{k(2-\alpha)})1_{[0,2^{l_1+4}]}(|\eta|)1_{[2^{l_2-4},2^{l_2+4}]}(|\sigma|).
	\end{equation}
	\begin{equation}\label{region2}
		\left|\partial_\eta m_2(\eta,\sigma)\right|\lesssim (2^{-m}2^{-l_1-l_2}2^{k(2-\alpha)})(2^{-m}2^{-l_2}2^{k(2-\alpha)})2^{-l_1}1_{[0,2^{l_1+4}]}(|\eta|)1_{[2^{l_2-4},2^{l_2+4}]}(|\sigma|).
	\end{equation}

	Considering $|k_1-k_3| \leqslant |k_1-k|+|k_3-k| \leqslant 8 $, we get
	
	\begin{align*}
		\left|F_{l_1,l_2,2}(\xi,s)\right|
		+\left|G_{l_1,l_2,2}(\xi,s)\right| 
		&\lesssim ||b||_{L^2}||g||_{L^\infty}||h||_{L^2}||\mathcal{F}^{-1}m_2||_{L^1}
		\\
		&\lesssim \varepsilon^3_1 2^{-m} 2^{-p_0m}2^{-10k^+}.
	\end{align*}

	Noticing the estimate \eqref{region2}, we could estimate $J_{l_1,l_2,2}(\xi,s)$ by
	\begin{align*}
		\left|J_{l_1,l_2,2}(\xi,s)\right| & \lesssim
		2^{l_1}2^{l_2}||\hat{f}_{k_1}||_{L^\infty}||\hat{f}_{k_2}||_{L^\infty}||\hat{f}_{k_3}||_{L^\infty}
		||\partial_\eta m_2||_{L^\infty}
		\\
		&\lesssim \varepsilon^3_1 2^{l_1+l_2}2^{-30k^+}(2^{-m}2^{-l_1-l_2}2^{k(2-\alpha)})^2
		\\
		&\lesssim \varepsilon^3_1 2^{-m} 2^{-p_0m}2^{-10k^+}.
	\end{align*}
	\textit{Case 2}: $l_2 \geq \max\{l_1,\bar{l}+1\}$.

	We could get the desire estimates using the same methods as in Case 1.  Therefore, we have
	\begin{equation*}
		\left|J_{l_1,l_2}(\xi,s)\right| \lesssim 2^{-m}\varepsilon_1^3 2^{-3p_0m}2^{-10k^+}.
	\end{equation*}
	\textit{Case 3}: $l_1=l_2=\bar{l}$.

	Using the decomposition \eqref{decom},  it suffices to prove that
	\begin{equation}\label{Jlbar}
		\left|J_{\bar{l},\bar{l}}(\xi,s)-c_0\frac{|\xi|^{2-\alpha}\hat{f}_{k_1}(\xi,s)\hat{f}_{k_2}(\xi,s)\hat{\bar{f}}_{k_3}(-\xi,s)}{s+1}\right|
		\lesssim \varepsilon^3_1 2^{-m}2^{-p_0m}2^{-10k_+}.
	\end{equation}
	By Taylor expansion, we get
	\begin{equation*}
		\left|\Phi(\xi,\eta,\sigma)-\alpha(\alpha-1)\frac{\eta \sigma}{|\xi|^{2-\alpha}}\right| \lesssim 2^{k(\alpha-3)} \left(|\eta|+|\sigma|\right)^3.
	\end{equation*}
	Therefore, by the $L^\infty$ bounds in \eqref{a13}, we have
	\begin{align}\label{J-J'}
		\left|J_{\bar{l},\bar{l}}(\xi,s)-J'_{\bar{l},\bar{l}}(\xi,s)\right|
		&\lesssim ||e^{\textit{i}s\Phi(\xi,\eta,\sigma)}-e^{
			\frac{\textit{i}\alpha(\alpha-1)s\eta\sigma}
			{|\xi|^{2-\alpha}}}||_{L^\infty}||\hat{f}_{k_1}||_{L^\infty}||\hat{f}_{k_2}||_{L^\infty}||\hat{f}_{k_3}||_{L^\infty}2^{2\bar{l}}
		\nonumber
		\\
		&\lesssim
		\varepsilon^3_1 2^m 2^{k(\alpha-3)}2^{5\bar{l}}2^{-30k^+},
	\end{align}
	where
	\begin{equation}
		J'_{\bar{l},\bar{l}}(\xi,s)=\int_{\mathbb{R}^2}e^{\frac{\textit{i}\alpha(\alpha-1)s\eta\sigma}{|\xi|^{2-\alpha}}}\hat{f}_{k_1}(\xi+\eta,s)\hat{f}_{k_2}(\xi+\sigma,s)\hat{\bar{f}}_{k_3}(-\xi-\eta-\sigma,s)\varphi(2^{-\bar{l}}\eta)\varphi(2^{-\bar{l}}\sigma)d\eta d\sigma,
	\end{equation}
	and
	\begin{align*}
		& \left|\hat{f}_{k_1}(\xi+\eta,s)\hat{f}_{k_2}(\xi+\sigma,s)\hat{\bar{f}}_{k_3}(-\xi-\eta-\sigma,s)-\hat{f}_{k_1}(\xi,s)\hat{f}_{k_2}(\xi,s)\hat{\bar{f}}_{k_3}(-\xi,s)\right|
		\\
		& \lesssim \varepsilon^3_1 2^{\frac{\bar{l}}{2}}2^{-20k_+}2^{p_0m}2^{-k},
	\end{align*}
	provided  $|\eta|+|\sigma| \leq 2^{\bar{l}+4}$.

	Therefore, we have
	\begin{align}\label{211}
		&\left|J'_{\bar{l},\bar{l}}(\xi,s)-\int_{\mathbb{R}^2}e^{\frac{\textit{i}\alpha(\alpha-1)s\eta\sigma}{|\xi|^{2-\alpha}}}\hat{f}_{k_1}(\xi,s)\hat{f}_{k_2}(\xi,s)\hat{\bar{f}}_{k_3}(-\xi,s)\varphi(2^{-\bar{l}}\eta)\varphi(2^{-\bar{l}}\sigma)d\eta d\sigma \right|
		\nonumber
		\\
		& \lesssim \varepsilon^3_1 2^{2\bar{l}} 2^{\frac{\bar{l}}{2}}\cdot 2^{-20k_+}2^{p_0m}2^{-k}
		\nonumber
		\\
		& \lesssim \varepsilon^3_1 2^{-1.1m}2^{-10k^+},
	\end{align}
	which follows the definition of $\bar{l}$ in \eqref{lbar}.

	On the other hand, one can show
	\begin{equation*}
		\int_{\mathbb{R}}e^{-ax^2-bx}dx=e^{\frac{b^2}{4a}}\frac{\sqrt{\pi}}{\sqrt{a}}, \ \ \ a,b \in \mathbb{C}, \ \text{Re} \ a>0
	\end{equation*}
	and
	\begin{align*}
		\int_{\mathbb{R}^2}e^{-ixy}e^{-\frac{x^2}{N^2}}e^{-\frac{y^2}{N^2}}dxdy&=\sqrt{\pi}N\int_{\mathbb{R}}e^{-\frac{y^2}{N^2}}e^{-\frac{y^2N^2}{4}}dy
		\\
		&=2\pi +O(N^{-1}).
	\end{align*}

	Recalling also that $2^{\bar{l}}\approx |\xi|^{1-\frac{\alpha}{2}}2^{-\frac{m+1}{2}}$, we get
	\begin{equation}\label{40000}
		\left|\int_{\mathbb{R}^2}e^{\frac{\textit{i}s\eta\sigma\alpha(\alpha-1)}{|\xi|^{2-\alpha}}}\varphi(2^{-\bar{l}}\eta)\varphi(2^{-\bar{l}}\sigma)d\eta d\sigma-2\pi\frac{1}{\alpha(1-\alpha)}\frac{|\xi|^{2-\alpha}}{s} \right|
		\lesssim 2^{k(2-\alpha)}2^{-m}2^{-p_0m}.
	\end{equation}

	Using (4.7), (4.8), (4.9) and \eqref{40000}, it's easy  to obtain
	\begin{align}\label{200}
		&|\int_{\mathbb{R}^2}e^{\frac{\textit{i}s\eta\sigma\alpha(\alpha-1)}{|\xi|^{2-\alpha}}}\hat{f}_{k_1}(\xi,s)\hat{f}_{k_2}(\xi,s)\hat{\bar{f}}_{k_3}(-\xi,s)\varphi(2^{-\bar{l}}\eta)\varphi(2^{-\bar{l}}\sigma)d\eta d\sigma
		\\ \nonumber
		&
		\quad-c_0\frac{|\xi|^{2-\alpha}\hat{f}_{k_1}(\xi,s)\hat{f}_{k_2}(\xi,s)\hat{\bar{f}}_{k_3}(-\xi,s)}{s}|
		\nonumber
		\\
		& \lesssim \varepsilon^3_1 2^{-m}2^{-10k^+}2^{-p_0 m}.
	\end{align}
	
	Combing  \eqref{J-J'}, \eqref{211} and \eqref{200}, we have finished the proof of \eqref{Jlbar}.
\end{proof}
\begin{Lemma}
	The bound \eqref{final} holds provided that
	\begin{align}
		&k_1,k_2,k_3 \in \left[-\frac{3}{8\alpha}m, \frac{53p_0+1}{7}m \right]
		\nonumber
		\\
		& \max \{|k_1-k|,|k_2-k|,|k_3-k| \} \geq 16,
		\nonumber
		\\
		&k_3-k_1 \geq 5.
	\end{align}
\end{Lemma}
\begin{proof}
	To achieve this goal, it suffices to prove that, for any $s\in [t_1,t_2]$,
	\begin{align}
		\left|I_{k_1,k_2,k_3}(s) \right|
		& \lesssim \varepsilon^3_1 2^{-p_0m}2^{-10k^+}2^{-m}.
	\end{align}

	Consider $k_3-k_1\geq 5$, and notice that
	\begin{align}\label{300}
		\left|\partial_\eta \Phi(\xi,\eta,\sigma)\right|&=\alpha sgn(\xi+\eta+\sigma)|\xi+\eta+\sigma|^{\alpha-1}-|\xi+\eta|^{\alpha-1}sgn(\xi+\eta)
		\nonumber
		\\
		& \gtrsim 2^{k_1(\alpha-1)},
	\end{align}
	provided that $|\xi+\eta|\approx 2^{k_1}$, $|\xi+\eta+\sigma| \approx 2^{k_3}$. We integrate by parts in $\eta$ to estimate
	\begin{equation*}
		|I_{k_1,k_2,k_3}(\xi,s)| \leq |J_{1}(\xi,s)|+|F_{1}(\xi,s)|+|G_{1}(\xi,s)|,
	\end{equation*}
	where
	\begin{align}
		&J_{1}(\xi,s)=\int_{\mathbb{R}^2}e^{is\Phi(\xi,\eta,\sigma)}\hat{f}_{k_1}(\xi+\eta,s)\hat{f}_{k_2}(\xi+\sigma,s)\hat{\bar{f}}_{k_3}(-\xi-\eta-\sigma,s)
		(\partial_\eta m_3)(\eta,\sigma)d\eta d\sigma
		\nonumber
		\\
		& F_{1}(\xi,s)=\int_{\mathbb{R}^2}e^{is\Phi(\xi,\eta,\sigma)}\partial\hat{f}_{k_1}(\xi+\eta,s)\hat{f}_{k_2}(\xi+\sigma,s)\hat{\bar{f}}_{k_3}(-\xi-\eta-\sigma,s) m_3(\eta,\sigma)d\eta d\sigma
		\\
		& G_{1}(\xi,s)=\int_{\mathbb{R}^2}e^{is\Phi(\xi,\eta,\sigma)}\hat{f}_{k_1}(\xi+\eta,s)\hat{f}_{k_2}(\xi+\sigma,s)\partial\hat{\bar{f}}_{k_3}(-\xi-\eta-\sigma,s) m_3(\eta,\sigma)d\eta d\sigma
		\nonumber
	\end{align}
	and
	\begin{equation*}
		m_3(\eta,\sigma)=\frac{1}{s\partial_\eta\Phi(\xi,\eta,\sigma)}\varphi_{[k_1-1,k_1+1]}(\xi+\eta)\varphi_{[k_3-1,k_3+1]}(\xi+\eta+\sigma).
	\end{equation*}

	Using \eqref{300}, it follows that
	\begin{equation*}
		||\mathcal{F}^{-1}m_3||_{L^1} \lesssim 2^{-m}2^{k_1(1-\alpha)},
	\end{equation*}
	\begin{equation*}
		||\mathcal{F}^{-1} \partial_\eta m_3||_{L^1} \lesssim 2^{-m}2^{-k_1\alpha}.
	\end{equation*}
	Therefore, we have
	\begin{align*}
		|J_{1}(\xi,s)| 
		&\lesssim  \varepsilon_1 2^{-\frac{m}{2}} \cdot \varepsilon_1 2^{p_0m}2^{-N_0k_3^+} \cdot \varepsilon_1 2^{p_0m}2^{-N_0k_2^+} \cdot 2^{-m}2^{-k_1\alpha}
		\\
		&\lesssim \varepsilon^3_1 2^{-\frac{m}{2}} 2^{2p_0m} 2^{-N_0k_3^+} 2^{-m}2^{-k_1\alpha}
		\\
		&\lesssim \varepsilon^3_1 2^{-m} 2^{10k^+} 2^{-p_0m}.
	\end{align*}
	\begin{align*}
		\left|F_{1}(\xi,s) \right|
		&\lesssim \varepsilon_1 2^{-k_1}  2^{p_0m} \cdot \varepsilon_1 2^{-\frac{m}{2}} \cdot \varepsilon_12^{-N_0k_3^+2^{p_0m}} \cdot 2^{-m}2^{k_1(1-\alpha)}
		\\
		&\lesssim \lesssim \varepsilon^3_1 2^{-m} 2^{10k^+} 2^{-p_0m},
	\end{align*}
	\begin{align*}
		|G_{1}(\xi,s)|
	&	\lesssim \varepsilon_1 2^{-\frac{m}{2}} \cdot \varepsilon_1 2^{p_0m}2^{-N_0\max{k_1^+,k_2^+}} \cdot \varepsilon_1 2^{-k_3}2^{p_0m} \cdot 2^{-m}2^{k_1(1-\alpha)}.
	\end{align*}

	As a result, we get
	\begin{equation*}
		|J_{2}(\xi,s)|+|F_{2}(\xi,s)|+|G_{2}(\xi,s)|
		\lesssim \varepsilon^3_1 2^{-m}2^{-p_0m}2^{-10k^+},
	\end{equation*}
	if $k_1,k_2,k_3 \in [-\frac{3}{8\alpha}m, \frac{53p_0+1}{7}m]$ and $|k| \leq 5 p_0 m$.
\end{proof}

\begin{Lemma}
	The bound \eqref{final} holds provided that
	\begin{align}
		&k_1,k_2,k_3 \in \left[-\frac{3}{8\alpha}m, \frac{53p_0+1}{7}m \right]
		\nonumber
		\\
		& \max \{|k_1-k|,|k_2-k|,|k_3-k|\} \geq 16,
		\nonumber
		\\
		&k_3-k_1 \leq 4.
	\end{align}
\end{Lemma}
\begin{proof}
	Assume that $k_1 \geq k+10$, and rewrite
	\begin{equation*}
		I_{k_1,k_2,k_3}(\xi,s)=\int_{\mathbb{R}^2 }e^{is\Phi(\xi,\eta,\sigma)}\hat{f}_{k_1}(\xi+\eta,s)\hat{f}_{k_2}(\xi+\sigma,s)\hat{\bar{f}}_{k_3}(-\xi-\eta-\sigma,s) \varphi_{[k_2-4,k_2-4]}(\sigma) d\eta d\sigma,
	\end{equation*}
	where $\Phi(\xi,\eta,\sigma)=|\xi|^\alpha-|\xi+\eta|^\alpha-|\xi+\sigma|^\alpha+|\xi+\eta+\sigma|^\alpha$.	
	We need prove 
	\begin{equation}
		|I_{k_1,k_2,k_3}(\xi,s)|\lesssim \varepsilon^3_1 2^{-m}2^{-p_0m}2^{-10k^+}.
	\end{equation}

	Consider $k_3\approx k_2 \approx k_1$, and notice that
	\begin{align}\label{300}
		|\partial_\eta \Phi(\xi,\eta,\sigma)|&=\alpha| sgn(\xi+\eta+\sigma)|\xi+\eta+\sigma|^{\alpha-1}-|\xi+\eta|^{\alpha-1}sgn(\xi+\eta)|
		\nonumber
		\\
		& \gtrsim 2^{k_2(\alpha-1)},
	\end{align}
	provided that $|\xi+\eta| \in [2^{k_1-2},2^{k_1+2}]$, $|\xi+\eta+\sigma| \in [2^{k_3-2},2^{k_3+2}]$, and $|\sigma| \in [2^{k_2-2},2^{k_2+2}]$.

	Integrating $I_{k_1,k_2,k_3}$ by parts in $\eta$ gives
	\begin{equation*}
		|I_{k_1,k_2,k_3}(\xi,s)| \leq |J_{2}(\xi,s)|+|F_{2}(\xi,s)|+|G_{2}(\xi,s)|,
	\end{equation*}
	where
	\begin{align}
		&J_{2}(\xi,s)=\int_{\mathbb{R}^2}e^{is\Phi(\xi,\eta,\sigma)}\hat{f}_{k_1}(\xi+\eta,s)\hat{f}_{k_2}(\xi+\sigma,s)\hat{\bar{f}}_{k_3}(-\xi-\eta-\sigma,s)
		(\partial_\eta m_4)(\eta,\sigma)d\eta d\sigma
		\nonumber
		\\
		& F_{2}(\xi,s)=\int_{\mathbb{R}^2}e^{is\Phi(\xi,\eta,\sigma)}\partial\hat{f}_{k_1}(\xi+\eta,s)\hat{f}_{k_2}(\xi+\sigma,s)\hat{\bar{f}}_{k_3}(-\xi-\eta-\sigma,s) m_4(\eta,\sigma)d\eta d\sigma
		\\
		& G_{2}(\xi,s)=\int_{\mathbb{R}^2}e^{is\Phi(\xi,\eta,\sigma)}\hat{f}_{k_1}(\xi+\eta,s)\hat{f}_{k_2}(\xi+\sigma,s)\partial\hat{\bar{f}}_{k_3}(-\xi-\eta-\sigma,s) m_4(\eta,\sigma)d\eta d\sigma
		\nonumber
	\end{align}
	and
	\begin{equation*}
		m_4(\eta,\sigma)=\frac{1}{s\partial_\eta\Phi(\xi,\eta,\sigma)}\varphi_{[k_1-1,k_1+1]}(\xi+\eta)\varphi_{[k_3-1,k_3+1]}(\xi+\eta+\sigma)\varphi_{[k_2-4,k_2+4]}(\sigma).
	\end{equation*}

	Using \eqref{300}, it follows that
	\begin{equation*}
		||\mathcal{F}^{-1} m_4||_{L^1} \lesssim 2^{-m}2^{k_2(1-\alpha)},
	\end{equation*}
	\begin{equation*}
		||\mathcal{F}^{-1} \partial_\eta m_4||_{L^1} \lesssim 2^{-m}2^{-k_2\alpha},
	\end{equation*}
	Therefore, by passing to the physical space and estimating it by Minkovski-inequality, we have
	\begin{align*}
		|J_{2}(\xi,s)|
		&
		\lesssim \varepsilon_1 2^{-\frac{m}{2}} \cdot \varepsilon_1 2^{p_0m}2^{-N_0k_2^+} \cdot \varepsilon_1 2^{p_0m}2^{-N_0k_3^+} \cdot 2^{-m}2^{-k_2\alpha}
		\\
		&
		\lesssim \varepsilon^3_1 2^{-m} 2^{p_0m}2^{-10k^+},
	\end{align*}
	\begin{align*}
		|F_{2}(\xi,s)|
		&\lesssim \varepsilon_1 2^{-k_1}2^{p_0m} \cdot  \varepsilon_1 2^{-\frac{m}{2}} \cdot \varepsilon_1 2^{p_0m}2^{-N_0k_3^+}\cdot 2^{-m}2^{k_2(1-\alpha)}
		\\
		&\lesssim \varepsilon^3_1 2^{-m} 2^{p_0m}2^{-10k^+},
	\end{align*}
	\begin{align*}
		|G_{2}(\xi,s)|
		&\lesssim \varepsilon_1 2^{-\frac{m}{2}} \cdot \varepsilon_1 2^{p_0m}2^{-N_0 \max\{k_1^+,k_2^+\}} \cdot \varepsilon_1 2^{-k_3} 2^{p_0m}\cdot 2^{-m}2^{k_2(1-\alpha)}
		\\
		&\lesssim \varepsilon^3_1 2^{-m} 2^{p_0m}2^{-10k^+}.
	\end{align*}
	Hence, we conclude that
	\begin{align*}
		|J_{2}(\xi,s)|+|F_{2}(\xi,s)|+|G_{2}(\xi,s)| 
		&\lesssim
		\varepsilon^3_1 2^{-m}2^{-p_0m}2^{-10k^+}.
	\end{align*}
	when $k_1,k_2,k_3 \in [-\frac{3}{8\alpha}m, \frac{53p_0+1}{7}m]$, $2^{k_1} \approx 2^{k_2} \approx 2^{k_3}$, and $|k|\leq 5p_0m$.
\end{proof}
\begin{Lemma}
	The bound \eqref{final} holds provided that
	\begin{align}
		&k_1 \in \left[-4m, -\frac{3}{8\alpha}m \right], \ k_1 +k_2 \leq -1.2m,
		\nonumber
		\\
		&k_2,k_3 \in \left[-4m, \frac{53p_0+1}{7}m \right],
		\nonumber
		\\
		& \max \{|k_1-k|,|k_2-k|,|k_3-k|\} \geq 16.	
	\end{align}
\end{Lemma}
\begin{proof}
	In fact, we have
	\begin{align*}
		|\int^{t_2}_{t_1}e^{\textit{i}H(\xi,s)}I_{k_1,k_2,k_3}(\xi,s)ds|
		& \lesssim \sup_{s}2^m \int_{\mathbb{R}^2}|
		\hat{f}_{k_1}(\xi-\eta,s)||\hat{f}_{k_2}(\eta-\sigma,s)||\hat{\bar{f}}_{k_3}(\sigma,s)|d\eta d\sigma
		\\
		&\lesssim  2^m 2^{k_1+k_2}
		 \cdot \varepsilon^3_1 2^{-10(k_1^+ +k_2^+ +k_3^+)}
		\\
		& \lesssim \varepsilon^3_1 2^{-10k^+}2^{-0.2m}
	\end{align*}
	and
	\begin{align*}
		\int^{t_2}_{t_1}e^{iH(\xi,s)}\frac{|\xi|^{2-\alpha}\hat{f}_{k_1}(\xi,s)\hat{f}_{k_2}(\xi,s)\hat{\bar{f}}_{k_3}(-\xi,s)}{s+1}ds=0.
	\end{align*}
	Then we finish our proof of this Lemma.
\end{proof}
\begin{Lemma}\label{11}
	The bound \eqref{final} holds provided that
	\begin{align}
		&k_1 \in \left[-4m, -\frac{3}{8\alpha}m\right],  \ k_1 +k_2 \geq -1.2m,
		\nonumber
		\\
		&k_2,k_3 \in \left[-4m, \frac{53p_0+1}{7}m \right],
		\nonumber
		\\
		& \max \{|k_1-k|,|k_2-k|,|k_3-k|\} \geq 16.	
	\end{align}
\end{Lemma}
\begin{proof}
	It suffices to prove
	\begin{align}\label{400}
		&\big|\int_{\mathbb{R}^2\times[t_1,t_2] }e^{iH(\xi,s)}e^{is\Phi(\xi,\eta,\sigma)}\hat{f}_{k_1}(\xi+\eta,s)\hat{f}_{k_2}(\xi+\sigma,s)\hat{\bar{f}}_{k_3}(-\xi-\eta-\sigma,s) d\eta d\sigma ds \big|
		\nonumber
		\\
		& \lesssim \varepsilon^3_1 2^{-p_0m}2^{-10k^+},
	\end{align}
	where
	\begin{equation*}
		\Phi(\xi,\eta,\sigma)=|\xi|^\alpha-|\xi+\eta|^\alpha-|\xi+\sigma|^\alpha+|\xi+\eta+\sigma|^\alpha,
	\end{equation*}
	\begin{equation*}
		H(\xi,s)=c_0|\xi|^{2-\alpha}\int^s_0|\hat{f}(\xi,r)|^2\frac{dr}{r+1}.
	\end{equation*}
	Set
	\begin{equation*}
		M(\xi)=\sum_{j=0}^2 M_j(\xi),
	\end{equation*}
	where
	\begin{align*}
		M_0(\xi)=&\int^{t_2}_{t_1}|\int_{\mathbb{R}^2 }e^{is\Phi(\xi,\eta,\sigma)}e^{iH(\xi,s)}\partial_s[\frac{1}{\Phi(\xi,\eta,\sigma)+\partial_s H(\xi,s)}\hat{f}_{k_1}(\xi+\eta,s)
		\\
		& \quad \hat{f}_{k_2}(\xi+\sigma,s)\hat{\bar{f}}_{k_3}(-\xi-\eta-\sigma,s)] d\eta d\sigma ds,
	\end{align*}
	\begin{align*}
		M_1(\xi)=&|\int_{\mathbb{R}^2 }e^{it_1\Phi(\xi,\eta,\sigma)}e^{iH(\xi,t_1)}[\frac{1}{\Phi(\xi,\eta,\sigma)+\partial_s H(\xi,t_1)}\hat{f}_{k_1}(\xi+\eta,t_1)
		\\
		& \quad \hat{f}_{k_2}(\xi+\sigma,t_1)\hat{\bar{f}}_{k_3}(-\xi-\eta-\sigma,t_1)] d\eta d\sigma|,
	\end{align*}
	\begin{align*}
		M_2(\xi)=&|\int_{\mathbb{R}^2 }e^{it_2\Phi(\xi,\eta,\sigma)}e^{iH(\xi,t_2)}[\frac{1}{\Phi(\xi,\eta,\sigma)+\partial_s H(\xi,t_2)}\hat{f}_{k_1}(\xi+\eta,t_1)
		\\
		& \quad \hat{f}_{k_2}(\xi+\sigma,t_2)\hat{\bar{f}}_{k_3}(-\xi-\eta-\sigma,t_2)] d\eta d\sigma|,
	\end{align*}

	Let
	\begin{align*}
		m_5(\eta,\sigma)=&\frac{1}{\Phi(\xi,\eta,\sigma)+\partial_s H(\xi,s)}\varphi_{[k_1-1,k_1+1]}(\xi+\eta)
		\\
		&\qquad \qquad \qquad \qquad \qquad \cdot \varphi_{[k_2-1,k_2+1]}(\xi+\sigma)\varphi_{[k_3-1,k_3+1]}(\xi+\eta+\sigma).
	\end{align*}
	Observing that
	\begin{equation*}
		|\Phi(\xi,\eta,\sigma)| \gtrsim 2^{\frac{k_2}{2}},
	\end{equation*}
	\begin{equation*}
		|\partial_s H(\xi,s)|\lesssim \varepsilon_1^2 2^{k(2-\alpha)}2^{-m}2^{-20k^+},
	\end{equation*}
	we get
	\begin{equation*}
		||\mathcal{F}^{-1}m_5||_{L^1} \lesssim 2^{-\frac{k_2}{2}}.
	\end{equation*}

	For $j=1,2$,  apply Lemma \ref{m} with
	\begin{equation*}
		\hat{b}(\eta)=e^{-it_j|\xi+\eta|^\alpha}\hat{f}_{k_1}(\xi+\eta,t_j),
	\end{equation*}
	\begin{equation*}
		\hat{g}(\sigma)=e^{-it_j|\xi+\sigma|^\alpha} \hat{f}_{k_2}(\xi+\sigma,t_j),
	\end{equation*}
	\begin{equation*}
		\hat{h}(-\eta-\sigma)=e^{it_j|\xi+\eta+\sigma|^\alpha}\hat{\bar{f}}_{k_3}(-\xi-\eta-\sigma,t_j).
	\end{equation*}
	and use (4.7)-(4.9) to conclude
	\begin{align*}
		M_j(\xi,s)
		&\lesssim \varepsilon_1 2^{-\frac{m}{2}}   \cdot \varepsilon_1 2^{-\frac{m}{2}} \cdot \varepsilon_1 2^{-\frac{m}{2}} \cdot 2^{-\frac{k_2}{2}} \cdot 2^{\frac{k_2}{2}} \cdot2^{\frac{k_2}{2}}
		\nonumber
		\\
		&\lesssim \varepsilon_1^3 2^{\frac{k_1}{2}}2^{-\frac{3}{2}m}
		\\
		&\lesssim \varepsilon_1^3 2^{-m}2^{-10k^+}2^{-p_0m}.
	\end{align*}

	As for $M_0$, define
	\begin{equation*}
		M_0(\xi,s)= \int^{t_2}_{t_1}\sum_{j=0}^3M^j_0(\xi,s),
	\end{equation*}
	where
	\begin{align*}
		M^0_0(\xi,s)=& |\int_{\mathbb{R}^2 }e^{is\Phi(\xi,\eta,\sigma)}\frac{\partial^2_sH(\xi,s)}{(\Phi(\xi,\eta,\sigma)+\partial_s H(\xi,s))^2}\hat{f}_{k_1}(\xi+\eta,s)
		\\
		& \quad \hat{f}_{k_2}(\xi+\sigma,s)\hat{\bar{f}}_{k_3}(-\xi-\eta-\sigma,s)d\eta d\sigma|
	\end{align*}
	\begin{align*}
		M^1_0(\xi,s)=\left|\int_{\mathbb{R}^2 }e^{is\Phi(\xi,\eta,\sigma)}m_5(\xi,\sigma)\partial_s\hat{f}_{k_1}(\xi+\eta,s)\hat{f}_{k_2}(\xi+\sigma,s)\hat{\bar{f}}_{k_3}(-\xi-\eta-\sigma,s)d\eta d\sigma \right|
	\end{align*}
	\begin{align*}
		M^2_0(\xi,s)=\left|\int_{\mathbb{R}^2 }e^{is\Phi(\xi,\eta,\sigma)}m_5(\xi,\sigma)\hat{f}_{k_1}(\xi+\eta,s)\partial_s\hat{f}_{k_2}(\xi+\sigma,s)\hat{\bar{f}}_{k_3}(-\xi-\eta-\sigma,s)d\eta d\sigma \right|
	\end{align*}
	\begin{align*}
		M^3_0(\xi,s)=\left|\int_{\mathbb{R}^2 }e^{is\Phi(\xi,\eta,\sigma)}m_5(\xi,\sigma)\hat{f}_{k_1}(\xi+\eta,s)\hat{f}_{k_2}(\xi+\sigma,s)\partial_s\hat{\bar{f}}_{k_3}(-\xi-\eta-\sigma,s)d\eta d\sigma \right|
	\end{align*}
	Therefore, by Lemma 3.3, we have
	\begin{align}\label{500}
		|M^1_0(\xi,s)| 
		& \lesssim   \varepsilon_1 2^{3p_0m}2^{-20k_1^+}2^{-m}   \cdot \varepsilon_1 2^{-\frac{m}{2}} \cdot \varepsilon_1 \varepsilon_1 2^{p_0 m}2^{-N_0k_3^+} \cdot 2^{-\frac{k_2}{2}} 
		\nonumber
		\\
		&\lesssim \varepsilon_1^3 2^{-m}2^{-p_0m}2^{-10k^+},
	\end{align}
	where we need $8\alpha <15$.
	Hence, we could get the following estimates
	\begin{align}\label{501}
		|M^2_0(\xi,s)| &
		\lesssim \varepsilon_1^3 2^{-m}2^{-p_0m}2^{-10k^+},
	\end{align}
	\begin{align}\label{502}
		|M^3_0(\xi,s)| &
		\lesssim \varepsilon_1^3 2^{-m}2^{-p_0m}2^{-10k^+}.
	\end{align}

	On the other hand, we have
	\begin{equation*}
		\sup_{s \in [t_1,t_2]}|\partial^2_s H(\xi,s)| \lesssim \varepsilon^2_1 2^{(2-\alpha)k}2^{-20k^+}2^{3p_0m-1.5m},
	\end{equation*}
	Combining with (4.7)-(4.9), it follows that
	\begin{align}\label{503}
		|M^0_0(\xi,s)|
		&\lesssim \varepsilon^3_1 2^{(2-\alpha)k}2^{-20k^+}2^{3p_0m-1.5m} 2^{3p_0m-1.5m}\cdot 2^{k_1+k_2}
		\nonumber
		\\
		& \lesssim \varepsilon^3_1 2^{-m}2^{-p_0m}2^{-10k^+},
	\end{align}
	for $k_1 \in [-4m, -\frac{3}{8\alpha}m],k_2 \in [-4m, \frac{53p_0+1}{7}m]$.

	In view of \eqref{500}-\eqref{503}, we complete the proof of this Lemma.
\end{proof}

\section*{Acknowlegement}
The first author is supported by Education Department of Hunan Province, general Program(grant No. 17C0039); the State Scholarship Fund of China Scholarship Council (No. 201808430121) and Hunan Provincial Key Laboratory of Intelligent Processing of Big Data on Transportation, Changsha University of Science and Technology, Changsha; 410114, China.

\end{document}